\documentclass{amsart}
\usepackage{graphicx}
\usepackage{amssymb,amsmath,amsfonts,amsthm,amsxtra,setspace,mathrsfs}
\usepackage{hyperref}
\usepackage{graphics}
\usepackage{comment}
\usepackage{mathtools}
\usepackage{harpoon}
\usepackage{enumerate}
\usepackage{xcolor}
\usepackage{comment}
\usepackage[margin=1.25in]{geometry}

\usepackage{tikz-cd}
\usepackage{tikz}
\usetikzlibrary{arrows.meta,decorations.pathreplacing,decorations.markings}
\usepackage{textcomp}
\usepackage[pagewise]{lineno}

\newtheorem{thm}{Theorem}

\newtheorem{cor}[thm]{Corollary}
\newtheorem{lem}[thm]{Lemma}
\newtheorem{prop}[thm]{Proposition}

\theoremstyle{definition}
\newtheorem{defn}[thm]{Definition}
\newtheorem{rem}[thm]{Remark}
\newtheorem{example}[thm]{Example}
\newtheorem{examples}[thm]{Examples}

\renewcommand{\epsilon}{\varepsilon}
\renewcommand{\phi}{\varphi}

\newcommand{\seq}{\subseteq}

\newcommand{\bC}{\mathbb{C}}
\newcommand{\bN}{\mathbb{N}}
\newcommand{\bQ}{\mathbb{Q}}
\newcommand{\bR}{\mathbb{R}}
\newcommand{\bZ}{\mathbb{Z}}

\newcommand{\cB}{\mathcal{B}}
\newcommand{\cC}{\mathcal{C}}
\newcommand{\cD}{\mathcal{D}}
\newcommand{\cF}{\mathcal{F}}

\newcommand{\dist}{\mathrm{dist}}

\newcommand{\ul}[1]{\underline{#1}}

\newcommand{\id}{\mathrm{id}}
\newcommand{\im}{\mathrm{im}}

\renewcommand{\Im}{\mathrm{Im}}

\newcommand{\SL}{\mathrm{SL}}
\newcommand{\PSL}{\mathrm{PSL}}
\newcommand{\GL}{\mathrm{GL}}

\newcommand{\CW}{\mathsf{CW}}
\newcommand{\Spc}{\mathsf{Spc}}
\newcommand{\Grp}{\mathsf{Grp}}
\newcommand{\Ab}{\mathsf{Ab}}

\newcommand{\ppi}{\ul{\pi}}
\newcommand{\pH}{\ul{H}}
\newcommand{\Shape}{\mathrm{Shape}}
\newcommand{\Pro}{\mathrm{Pro}}
\newcommand{\Ind}{\mathrm{Ind}}
\newcommand{\Ho}{\mathrm{Ho}}
\newcommand{\Hom}{\mathrm{Hom}}
\newcommand{\Aut}{\mathrm{Aut}}
\newcommand{\ev}{\mathrm{ev}}

\newcommand{\UConf}{\mathrm{UConf}}

\newcommand{\norm}[1]{\left\Vert #1 \right\Vert}

\newcommand{\sgen}[1]{\left\langle #1 \right\rangle}
\newcommand{\wt}[1]{\widetilde{#1}}

\title{Continuous approximate roots of polynomial equations via shape theory}
\author[J. Lau]{Joshua Lau${}^1$} 
\address[J. Lau]{University of Toronto, Department of Mathematics, 40 St. George St., Toronto, Ontario, Canada, M5S 2E4}%
\email{jlau@math.utoronto.ca}
\urladdr{}
\thanks{${}^1$ Supported by an NSERC Postgraduate scholarship (PGS-D)}

\author[V. Marin-Marquez]{Vicente Marin-Marquez${}^{2}$} 
\address[V. Marin-Marquez]{University of Toronto, Department of Mathematics, 40 St. George St., Toronto, Ontario, Canada, M5S 2E4}%
\email{vicente.marinmarquez@mail.utoronto.ca}
\urladdr{}
\thanks{${}^2$ Supported by an Ontario Graduate Scholarship}
\date{\today}


\begin{document}

\begin{abstract}
We study continuous approximate solutions to polynomial equations over the ring $C(X)$ of continuous complex-valued functions over a compact Hausdorff space $X$. We show that when $X$ is one-dimensional, the existence of such approximate solutions is governed by the behaviour of maps from the fundamental pro-group of $X$ into braid groups. 
\end{abstract}

\maketitle
\section{Introduction}

Let $X$ be a compact Hausdorff space (henceforth a \emph{compactum}), and denote by $C(X)$ the C*-algebra of continuous complex-valued functions on $X$. For each positive integer $n$ and each $n$-tuple of continuous functions $a_0,a_1, \ldots, a_{n-1} \in C(X)$, we obtain a monic polynomial over the ring $C(X)$ of degree $n$ given by 
\begin{align*}
    P(x, z) = z^n + a_{n-1}(x)z^{n-1} + \cdots + a_1(x)z + a_0(x)
\end{align*}
One obvious question one can ask is whether $P(x,z)$ has any roots in the ring $C(X)$, that is, if there exist functions $f \in C(X)$ such that $P(x,f(x)) = 0$ for all $x \in X$. 

The problem of characterizing those topological spaces $X$ for which \emph{every} monic polynomial $P(x,z)$ over $C(X)$ has a global continuous solution has been extensively studied \cite{Countryman,KawamuraMiura,HatoriMiura2000,HonmaMiura2007}. Work of Gorin and Lin \cite{GorinLin} has also established necessary and sufficient conditions for $P(x,z)$ to split into linear factors over $C(X)$ in the case that $P(x_0,z)$ has no repeated roots for each $x_0 \in X$. 

While there has been considerable effort in understanding polynomials over $C(X)$, most of the results in the literature have been \emph{algebraic} in nature, concerning only properties $C(X)$ has as a ring. However, in \cite{Miura1999}, Miura takes an \emph{analytic} approach, making use of the sup-norm on $C(X)$. Together with Kawamura in \cite{KawamuraMiura}, this analytic approach is further investigated and for continua of dimension at most one, they obtain a topological characterization for when $C(X)$ possesses approximate $n^{\text{th}}$ roots. 

Recently, these analytic results have found a fruitful application in the continuous model theory of Banach algebras (following \cite{BenYaacovEtAl}). Indeed, in \cite{EagleLau} it was shown that the approximate algebraic closure of $C(X)$ is an $\forall\exists$-axiomatizable property in the language of unital C*-algebras and an explicit axiomatization is given. On the other hand, it is also shown that the property of (exact) algebraic closure of $C(X)$ is not axiomatizable in this language. 

In this paper, we provide topological characterizations for a continuum, that is, a connected compactum, to have an approximately algebraically closed $C(X)$, in the case that $X$ has covering dimension one. The most general result concerns a shape-theoretic invariant, the \emph{fundamental pro-group} $\ppi_1(X, x_0)$ of a pointed space $(X, x_0)$ (see Section \ref{sec:proCatShape} for a review of shape theory), and maps from $\ppi_1(X, x_0)$ into the braid group on $n$ strands $\mathcal{B}_n$.

\begin{thm}[Theorem \ref{thm:low_dim_approx_alg_closure_propi1}]\label{thm:MainThm}
    Suppose that $X$ is a continuum with covering dimension at most one. Consider any basepoint $x_0 \in X$. Then $C(X)$ is approximately algebraically closed if and only if for every $n \geq 1$, all morphisms $\ppi_1(X, x_0) \to \mathcal{B}_n$ map into the kernel of the canonical map $\mathcal{B}_n \to S_n$, i.e. all morphisms $\ppi_1(X, x_0) \to \mathcal{B}_n$ factor through the inclusion $N_n \to \mathcal{B}_n$ of the subgroup of pure braids.
\end{thm}

Using this criterion, we are able to provide examples of continua that are not (exactly) algebraically closed, but are approximately algebraically closed (Examples \ref{ex:variousExamplesofApproxAlgClosure_cechH1}). An important application of our results is to the study of co-existentially closed continua. This class of continua was first introduced by Bankston in \cite{Bankston1999}, and are known to be one-dimensional \cite[Corollary 4.13]{Bankston2006}, so our results apply (Example \ref{ex:co-ec_continua}). We believe that this is an important first step towards a classification of co-existentially closed continua. From the perspective of continuous model theory, the next step towards this goal would include studying more general classes of equations over one-dimensional spaces, namely, non-monic equations, multivariable equations, and *-polynomials (i.e.~polynomials in $z$ and $\overline{z}$). \par

The remainder of this paper is organized as follows. In Section \ref{sec:Preliminaries}, we set up and discuss some elementary facts concerning polynomials over $C(X)$. We also review some relevant facts about braid groups. Section \ref{sec:pertStab} is analytic in nature, as we relate approximate roots of an arbitrary polynomial to exact roots of better-behaved polynomials (i.e. those with non-vanishing discriminant). In Section \ref{sec:proCatShape}, we review some material on shape theory we will need, especially results about the fundamental pro-group, which is the primary topological invariant we use in this paper. In Section \ref{sec:polys_via_ppi1}, we formulate how exact roots of well-behaved polynomials are related to morphisms from the fundamental pro-group into braid groups. In Section \ref{sec:main_results}, we combine the results in the previous sections to obtain Theorem \ref{thm:MainThm} above, and provide examples of how to use it. Finally, in Section \ref{sec:lowdegpolys} we explain some simpler characterizations for obtaining approximate roots of low-degree polynomials, and give some examples to show that these simpler invariants cannot be used for higher degree polynomials. 

\subsection*{Acknowledgements}
We are grateful to Christopher Eagle and George Elliott for helpful discussions and comments. We would also like to thank the anonymous referee for helpful comments, and especially for pointing out a proof of Theorem \ref{thm:noRepeatRootsDense} that much simplified our original argument.

\numberwithin{thm}{section}
\section{Spaces of polynomials and their solutions}\label{sec:Preliminaries}
\subsection{Spaces of polynomials}\label{subsec:prelim_spaces_of_polys}
This paper is concerned with monic polynomials over the ring of continuous functions on a compactum. Given a compactum $X$, we write a monic polynomial $P \in C(X)[z]$ of degree $n \geq 1$ as
\begin{align*}
    P(x,z) = z^{n} + a_{n-1}(x)z^{n-1} + \cdots + a_1(x)z + a_0(x) 
\end{align*}
where $a_0, a_1, \ldots, a_{n-1} \in C(X)$. We typically denote polynomials over $C(X)$ using capital letters such as $P$ and $Q$, whereas we typically denote polynomials over $\mathbb{C}$ with lowercase letters such as $p$ and $q$. Observe that if $C(X,\mathbb{C}^n)$ denotes the Banach space of continuous functions from $X$ to $\mathbb{C}^n$, then we have a natural identification of vector spaces
\begin{align*}
    \{ P \in C(X)[z] : P \text{ is monic with degree $n$}\} &\longleftrightarrow C(X,\mathbb{C}^n) \\
    z^{n} + a_{n-1}(x)z^{n-1} + \cdots + a_1(x)z + a_0(x) &\longleftrightarrow (a_0, a_1, \ldots, a_{n-1})
\end{align*}
This identification, endows the space of monic polynomials of degree $n$ over $C(X)$ with the structure of a Banach space, and in particular this space is topologised with the norm 
\begin{align*}
    \norm{P} = \max_{0 \leq i < n} \norm{a_i}_{\infty}
\end{align*}
for $P(x,z) = z^{n} + a_{n-1}(x)z^{n-1} + \cdots + a_1(x)z + a_0(x)$. Due to this identification, we make no distinction between a monic polynomial of degree $n$ in $C(X)[z]$ and a continuous map $X \to \mathbb{C}^n$. \par

Observe that for each $x_0 \in X$ we have a natural evaluation map $\ev_{x_0} : C(X,\mathbb{C}^n) \to \mathbb{C}[z]$ given by 
\begin{align*}
    \ev_{x_0}(a_0, a_1, \ldots, a_{n-1}) = z^n +a_{n-1}(x_0)z^{n-1} + \cdots + a_{1}(x_0)z + a_0(x_0).
\end{align*}
For a monic polynomial $P(x,z) \in C(X)[z]$, we denote the monic polynomial $\ev_{x_0}(P) \in \mathbb{C}[z]$ by $P(x_0,z)$. Applying the usual evaluation map $\ev_{z_0} : \mathbb{C}[z] \to \mathbb{C}$ at a point $z_0 \in \mathbb{C}$, we obtain, for each $x_0 \in X$ and each $z_0 \in \mathbb{C}$, a complex number $P(x_0,z_0)$ given by $P(x_0,z_0) = \ev_{z_0} \ev_{x_0}(P)$. \par

Above, the discussion focuses solely on a monic polynomial in $C(X)[z]$ in terms of its coefficients, but there is an alternative and useful perspective of thinking in terms of the roots. This is done as follows: To each monic polynomial $P \in C(X)[z]$ of degree $n$, we can associate a map $X \to \mathbb{C}^n/S_n$ defined by associating to each $x_0 \in X$, the multiset of roots of of the complex polynomial $P(x_0,z) \in \mathbb{C}[z]$ (this is well-defined because the polynomial is monic and hence does not drop degree as $x_0$ varies). By a consequence of Rouch\'e's Theorem (see Lemma \ref{lem:continuousDependenceOfRoots}), this map is continuous. 

An especially important class of polynomials are those that never have repeated roots, in which case this map descends to a continuous map $X \to \UConf_n(\mathbb{C})$, the unordered configuration space on $n$ points in $\mathbb{C}$. This space of polynomials with no repeated roots also has a description in terms of the coefficients as follows. Let $\Delta : \mathbb{C}^n \to \mathbb{C}$ denote the \emph{monic discriminant function}, i.e. $\Delta(a_0, a_1, \ldots, a_{n-1})$ is the discriminant of the degree $n$ monic polynomial $z^n + a_{n-1}z^{n-1} + \cdots + a_1 z + a_0$. Then $\Delta$ is some polynomial in $a_0, a_1, \ldots, a_{n-1}$, and its zero-set defines the \emph{(monic) discriminant variety} 
$$
V(\Delta) = \{ (a_0, a_1, \ldots, a_{n-1}) \in \mathbb{C}^n : \Delta(a_0, a_1, \ldots, a_{n-1}) = 0 \}
$$
which is a closed subset of $\mathbb{C}^n$. Then the open subset $B_n := \mathbb{C}^n \setminus V(\Delta)$ of $\mathbb{C}^n$ is the space of coefficients for monic polynomials of degree $n$ having no repeated roots, so the space of degree $n$ monic polynomials over $C(X)$ having no repeated roots is identified with the subset $C(X,B_n)$ in $C(X, \mathbb{C}^n)$. \par


The space of monic polynomials over $\mathbb{C}$ exhibits a natural $\mathbb{C}^{\times}$-action given by scaling the roots. Precisely, if we denote this action by $\star$, then a number $\mu \in \mathbb{C}^\times$ acts on a polynomial $\prod_{1\leq i \leq n}(z - \lambda_j) \in \mathbb{C}[z]$ by
\begin{align*}
    \mu \star \prod_{1\leq i \leq n}(z - \lambda_j) = \prod_{1 \leq i \leq n}(z-\mu \lambda_j)
\end{align*}
In terms of the coefficients, this action is given by
\begin{align*}
    \mu \star (z^n + a_{n-1}z^{n-1} + \cdots + a_1 z + a_0) = z^n + \mu a_{n-1} z^{n-1} + \cdots + \mu^{n-1}a_1 z + \mu^n a_0
\end{align*}
The perspective of this action in terms of the roots shows that the action clearly restricts to the space $B_n$. An important observation is that with respect to this action, the discriminant function $\Delta : B_n \to \mathbb{C}^\times$ is homogeneous of degree $n(n-1)$: 
\[\Delta(\mu \star p) = \prod_{1 \leq i < j \leq n} \mu^2(\lambda_i - \lambda_j)^2 = \mu^{n(n-1)} \Delta(p)\]
This implies that $\Delta : B_n \to \mathbb{C}^\times$ is a fibre bundle for $n \geq 2$. Indeed, if $n\geq 2$, then we can find neighbourhoods $U$ and $V$ of $1 \in \bC^\times$ such that the map $\mu \mapsto \mu^{n(n - 1)}$ defines a homeomorphism from $U$ to $V$. Then given a polynomial $p_0 \in B_n$ and its discriminant $\delta_0 = \Delta(p_0)$ we have the trivialization map
\begin{align*}
    U \times \Delta^{-1}(\delta_0) & \longrightarrow \Delta^{-1}(V\delta_0) \\
    (\mu, p) & \longmapsto \mu \star p
\end{align*}
which is a homoemorphism as the inverse is given by taking some $q \in \Delta^{-1}(V\delta_0)$ to $(\mu, \mu^{-1} \star q)$ for $\mu = (\Delta(q)/\delta_0)^{1/n(n-1)}$, where the $n(n - 1)$th root is taken using the inverse $V \to U$ we specified above. This shows that $\Delta: B_n \to \bC^\times$ is a fibre bundle.
We define the subset $B_n' = \Delta^{-1}(1) \seq B_n$ to be the fibre over $1 \in \mathbb{C}^\times$, i.e. the space of monic, degree $n$ polynomials with no repeated roots having discriminant equal to $1$. This subset $B_n'$ is path-connected (\cite[Lemma 3.5]{GorinLin}).

\subsection{Continuous roots of polynomials}\label{subsec:spaces_of_polys_and_roots}
Let $X$ be a compactum. As mentioned in the introduction, the question of when a monic polynomial $P \in C(X)[z]$ has a root in the ring $C(X)$ (i.e. when there is a continuous function $f \in C(X)$ such that $P(x,f(x)) = 0$ for all $x\in X$) has been extensively studied. If $f \in C(X)$ is a root of $P$, we sometimes say $f$ is an \emph{exact root} to distinguish it from the approximate roots that are treated below. If $P$ factors completely over $C(X)$, i.e. when there exist $f_1, \ldots, f_{\deg P} \in C(X)$ such that $P(x,z) = \prod_{k=1}^{\deg P} (z-f_k(x))$, we say that $P$ is \emph{completely solvable}. Complete solvability of polynomials over $C(X)$ was studied by Gorin and Lin in \cite{GorinLin}, and we adapt many of their methods here. \par 

\begin{defn}\label{defn:alg_closed_compactum}
    For $X$ a compactum, we say that the ring $C(X)$ is \emph{algebraically closed} (resp. \emph{completely solvable}) if every non-constant monic polynomial over $C(X)$ has an exact root (resp. is completely solvable). 
\end{defn}

In this paper, we are also interested in \emph{approximate} roots of monic polynomials over $C(X)$. Let us make this notion precise. 

\begin{defn}\label{defn:approx_roots}
    Suppose that $P(x,z)$ is a monic polynomial over $C(X)$, and let $\epsilon > 0$ be given. We say that a function $f \in C(X)$ is an \emph{$\epsilon$-approximate root} of $P$ provided that $|P(x,f(x))| < \epsilon$ for all $x \in X$. We say that $P$ \emph{has approximate roots} if $P$ has an $\epsilon$-approximate root for all $\epsilon > 0$.
\end{defn}

\begin{defn}\label{defn:approx_alg_closed_compatum}
    Let $X$ be a compactum. We say that $C(X)$ is \emph{approximately algebraically closed} if every non-constant monic polynomial over $C(X)$ has approximate roots. 
\end{defn}

In the case when a polynomial $P$ of degree $n \geq 1$ has no repeated roots, there is a topological description for when $P$ admits a continuous root that lends itself well to homotopical methods. Following \cite{GorinLin}, we define the solution space for a monic polynomial of degree $n$ with no repeated roots to be the space
\[E_n = \{((a_0, a_1, \ldots, a_{n-1}), z_0) \in B_n \times \bC : \ev_{z_0}(a_0, a_1, \ldots, a_{n-1}) = 0\} \subset \mathbb{C}^n \times \mathbb{C}\]
which consists of elements in $B_n$ along with a choice of root (such a choice always exists by the Fundamental Theorem of Algebra). Then we have the natural projection map $\rho: E_n \to B_n$ that forgets the root, and we see that given a polynomial $P: X \to B_n$ a solution of $P$ is nothing but a continuous lift $\lambda: X \to E_n$ making the diagram 
\[\begin{tikzcd}
	& {E_n} \\
	X & {B_n}
	\arrow["\rho", from=1-2, to=2-2]
	\arrow["\lambda", dashed, from=2-1, to=1-2]
	\arrow["P"', from=2-1, to=2-2]
\end{tikzcd}\]
commute. The projection map $\rho: E_n \to B_n$ is an $n$-sheeted covering map, and so understanding the fundamental group of these spaces will be crucial for determining when we can find lifts (and hence roots of polynomials). As explained in \cite{GorinLin}, the fundamental group of $B_n$ is the Artin braid group on $n$-strands, which we denote by $\mathcal{B}_n$. Moreover, $B_n$ is an Eilenberg-MacLane space $K(\mathcal{B}_n, 1)$ \cite[Corollary 2.2]{FadellNeuwirth}. \par
The fundamental group $M_n$ of $E_n$ is then an index $n$ subgroup of $\pi_1(B_n) \cong \mathcal{B}_n$ and can be identified with the subgroup of braids that fix the first strand. Note that the way this group $M_n$ sits inside $E_n$ depends on a choice of basepoint. Choosing a basepoint $b_0 \in B_n$, we have $n$ choices for a basepoint of $E_n$ that maps to $b_0$, corresponding to the $n$ distinct choices of roots for the monic polynomial $b_0$; let's denote them $z_1, \dots, z_n$. Then the start and end points of the strands for a braid in $\pi_1(B_n, b_0) \cong \mathcal{B}_n$ are labeled by $z_1, \dots, z_n$ and the subgroup $M_{n, i} = \pi_1(E_n, (b_0, z_i))$ of $\mathcal{B}_n$ consists of the braids whose strand starting at $z_i$ also ends at $z_i$. \par
Similarly we can define the space $\widetilde{E}_n \subset \mathbb{C}^n \times \mathbb{C}^n$ whose elements consist of a polynomial in $B_n$ along with an ordering of all $n$ of its roots, that is, 
\[\widetilde{E}_n = \{((a_0, a_1, \ldots, a_{n-1}), z_1, \dots, z_n) \in B_n \times \bC^n : \ev_{z_i}(a_0, a_1, \ldots, a_{n-1}) = 0 \quad \forall i \in \{1, \dots, n\}  \}. \]
We also have a natural projection map $\widetilde{\rho}: \widetilde{E}_n \to B_n$ that forgets all $n$ roots, which is an $n!$-sheeted covering map as there are $n!$ ways to order the $n$ roots of a polynomial. The fundamental group $N_n$ of $\widetilde{E}_n$ is the pure braid group on $n$-strands, which corresponds to the subgroup of $B_n$ consisting of braids with each of its $n$ strands starting and ending at the same point (this is the kernel of the natural homomorphism $\tau: \cB_n \to S_n$ into the symmetric group). Analogously to lifting along $\rho$, a polynomial $P: X \to B_n$ is completely solvable if and only if there is a lift of $P$ along $\widetilde{\rho}$. Since $E_n$ and $\widetilde{E}_n$ are covering spaces of $B_n$, they are also Eilenberg--MacLane spaces. \par
A final important subgroup of $\cB_n$ is the commutator subgroup $\cB_n'$. As explained in \cite{GorinLin}, the quotient $\cB_n/\cB_n'$ is isomorphic to $\bZ$ and that the discriminant function $\Delta: B_n \to \bC^\times$ induces the abelianization map $\pi_1(\Delta): \cB_n \to \bZ$. Then using the fibre sequence
\[B_n' \longrightarrow B_n \longrightarrow \bC^\times\]
that we found at the end of Section \ref{subsec:prelim_spaces_of_polys}, we get the exact sequence
\[\begin{tikzcd}
	{0 \cong \pi_2(\bC^\times)} & {\pi_1(B_n')} & {\pi_1(B_n)} & {\pi_1(\bC^\times) \cong \bZ} & {\pi_0(B_n') \cong 0}
	\arrow[from=1-1, to=1-2]
	\arrow[from=1-2, to=1-3]
	\arrow["{\pi_1(\Delta)}", from=1-3, to=1-4]
	\arrow[from=1-4, to=1-5]
\end{tikzcd}\]
which gives us that $\pi_1(B_n') \cong \cB_n'$ (justifying the notation) and that the inclusion map on spaces $B_n' \to B_n$ induces the inclusion map on subgroups $\cB_n' \cong \pi_1(B_n') \to \pi_1(B_n) \cong \cB_n$.

\section{Stability and perturbation}\label{sec:pertStab}
The goal of this section is to reduce the study of approximate roots for all polynomials over $C(X)$ to the study of exact roots of polynomials over $C(X)$ that have no repeated roots. This will allow us to later use the methods in the paper \cite{GorinLin}, which treats the no repeated roots case. Thus, the main theorem of this section is the following:

\begin{thm}\label{thm:perturbAndStabDegN}
    Let $X$ be a compactum with $\dim X \leq 1$ and let $n$ be a positive integer. The following are equivalent:
    \begin{enumerate}
        \item Every monic polynomial of degree $n$ with coefficients in $C(X)$ has approximate roots;
        \item Every monic polynomial of degree $n$ with coefficients in $C(X)$ with no repeated roots has approximate roots;
        \item Every monic polynomial of degree $n$ with coefficients in $C(X)$ with no repeated roots has an exact root. 
    \end{enumerate}
\end{thm}

The basic strategy is as follows. First, to prove that (2) implies (3) in the above, we need to understand under what circumstances having roots is a \emph{stable} property, i.e. when a given polynomial over $C(X)$ having approximate roots actually implies that it has an exact root. Then, to show that (3) implies (1), we are led to investigate when we can \emph{perturb} the coefficients of a polynomial so that it has no repeated roots, and this is where the restriction on the dimension of $X$ appears. The key ingredient is the simple observation that the roots of a polynomial are bounded by the coefficients. In our case, this means the following. 

\begin{lem}\label{lem:polyRootBound}
    Let $X$ be a compactum, let $P(x,z) = z^n + a_{n-1}(x)z^{n-1} + \cdots + a_1(x)z + a_0(x)$ be a polynomial with coefficients in $C(X)$, and let $M := \max\{  \norm{a_0}, \norm{a_1}, \ldots, \norm{a_{n-1}} \}$. If $f \in C(X)$ is an $\epsilon$-approximate solution of $P$, then $\norm{f}_{\infty} \leq 1 + \epsilon + M$. 
\end{lem}
\begin{proof}
    Let $f \in C(X)$ be an $\epsilon$-approximate solution of $P$. Take any $x_0 \in X$, and let $w_0 := P(x_0,f(x_0))$. Since $f$ is an $\epsilon$-approximate solution of $P$, we know that $|w_0| < \epsilon$. Observe that 
    \begin{align*}
        z^n + a_{n-1}(x_0)z^{n-1} + \cdots + a_1(x_0)z + (a_0(x_0) - w_0)
    \end{align*}
    is a polynomial over $\mathbb{C}$ possessing $f(x_0)$ as a root. Hence by Cauchy's bound for the root of a polynomial, we have that 
    \begin{align*}
        |f(x_0)| &\leq 1 + \max\{ |a_0(x_0) - w_0|, |a_1(x_0)| , \ldots, |a_{n-1}(x_0)| \} \\
        &\leq 1 + \max\{ |a_0(x_0)| + |w_0|, |a_1(x_0)| , \ldots, |a_{n-1}(x_0)| \} \\
        &\leq 1 + |w_0| + \max\{ |a_0(x_0)|, |a_1(x_0)| , \ldots, |a_{n-1}(x_0)| \} \\
        &\leq 1 + \epsilon + \max\{ |a_0(x_0)|, |a_1(x_0)| , \ldots, |a_{n-1}(x_0)| \} \\
        &\leq 1 + \epsilon + M
    \end{align*}
    As the above bound holds for any $x_0 \in X$, we obtain $\norm{f}_{\infty} \leq 1 + \epsilon + M$. 
\end{proof}

As mentioned in Section \ref{subsec:prelim_spaces_of_polys}, we can think of any monic polynomial $P(x,z)$ over $C(X)$ in terms of its roots, by considering the map $X \to \mathbb{C}^n/S_n$ sending some $x_0 \in X$ to the multiset of roots of $P(x_0,z) \in \mathbb{C}[z]$. We take some time to formulate what this means precisely, and show that this map is continuous. 

\begin{lem}\label{lem:continuousDependenceOfRoots}
    The roots of a polynomial depend continuously on its coefficients. Precisely, given any monic polynomial $p(z) = z^n + a_{n - 1}z^{n - 1} + \cdots + a_1z + a_0$ that factors as
    \[p(z) = \prod_{i = 1}^r (z - \lambda_i)^{m_i}\]
    for some distinct roots $\lambda_i \in \bC$ with multiplicities $m_i$, and any $\delta \in (0, \frac{1}{2}\min_{i \neq j} |\lambda_i - \lambda_j|)$, there exists an $\epsilon > 0$ so that any monic polynomial $q(z) = z^n + b_{n - 1}z^{n - 1} + \cdots + b_1z + b_0$ with coefficients satisfying $\max_{0 \leq k < n} |a_k - b_k| < \epsilon$ has exactly $m_i$ roots counted with multiplicity in the disk $B_{\delta}(\lambda_i) = \{z \in \mathbb{C}: |z - \lambda_i| < \delta\}$.
\end{lem}
\begin{proof}
    For each $0 \leq i \leq r$ consider the circle $C_i = \{z \in \mathbb{C}: |z - \lambda_i| = \delta\}$, which by construction contains no root of $p(z)$. Let
    \[\mu = \min_{0 \leq i \leq r} \min_{z \in C_i} |p(z)|\]
    where the minima are achieved and positive by compactness of the circles and continuity of $p$. Then since the values of a polynomial over $\mathbb{C}$ depend continuously on its coefficients, we can find an $\epsilon > 0$ such that if $q(z) = z^n + b_{n - 1}z^{n - 1} + \cdots + b_1z + b_0$ is a polynomial with $|a_k - b_k| < \epsilon$ then
    \[\max_{0 \leq i \leq r} \max_{z \in C_i} |q(z) - p(z)| < \mu\]
    But then for any such $q(z)$, we have $|q(z) - p(z)| < |p(z)|$ for all $z \in C_i$, which by Rouch\'e's Theorem implies that $p(z)$ and $q(z)$ have the same number of zeroes in the disk $\{z \in \mathbb{C} : |z - \lambda_i| < \delta\}$ as required.
\end{proof}

We can reformulate the above lemma as follows. Given a monic polynomial $p \in \mathbb{C}[z]$, let us denote by $\sigma_p$ the multiset of roots of $p$. Here, we view $\sigma_p$ as an element of the quotient space $\mathbb{C}^n/S_n$, which we endow with the quotient topology. In these terms, the above lemma becomes the following. 

\begin{cor}\label{cor:continuityOfRootMultiset}
    The map
    \begin{align*}
        \sigma: \bC^n & \longrightarrow \bC^n/S_n \\
        (a_0, \dots, a_{n - 1}) & \longmapsto \sigma_{z^n + a_{n - 1}z^{n - 1} + \cdots + a_1z + a_0}
    \end{align*}
    taking the coefficients of a polynomial in $\mathbb{C}[z]$ to its multiset of roots is continuous. In particular, given a monic polynomial $P(x,z)$ with coefficients in $C(X)$, the map $X \to \mathbb{C}^n/S_n , x_0  \mapsto \sigma_{P(x_0,z)}$ is continuous. 
\end{cor}

For each complex number $z \in \mathbb{C}$ and each multiset $A \in \mathbb{C}^n/S_n$, we define the \emph{distance from $z$ to $A$} to be the number $d(z,A) := \min_{w \in A} |z-w|$. This assignment
\begin{align*}
    \bC \times (\bC^n/S_n) & \longrightarrow [0, \infty) \\
    (z, A) & \longmapsto d(z, A)
\end{align*}
is a continuous function. Using these facts, we now show that approximate roots of $P(x,z)$ are uniformly close to the multiset of roots of $P(x,z)$. This is reminiscent of the notion of a \emph{weakly stable predicate} in the sense of \cite[Definition 3.2.4]{FarahEtAl}. 

\begin{lem}\label{lem:approxSolutionsAreCloseToRoots}
    Let $X$ be a compactum, and let $P(x,z) = z^n + a_{n-1}(x)z^{n-1} + \cdots + a_1(x)z + a_0(x)$ be a monic polynomial with coefficients in $C(X)$. Then given any $\delta > 0$ there exists an $\epsilon > 0$ such that all $\epsilon$-approximate solutions $f \in C(X)$ of $P$ are within $\delta$ of one of the roots of $P(x, z)$, i.e. $d(f(x), \sigma_{P(x,z)}) < \delta$ for all $x \in X$. 
\end{lem}
\begin{proof}
    Let $R := 2 + \max_{0 \leq i < n} \norm{a_i}$ and consider the set
    \[Y := \{(x, w) \in X \times \bC \ : \ |w| \leq R, \quad d(w, \sigma_{P(x,z)}) \geq \delta\}\]
    Using Corollary \ref{cor:continuityOfRootMultiset} we see that the function $(x, w) \mapsto d(w, \sigma_{P(x,z)})$ is continuous, and hence $Y$ is a closed set. Being a subset of the compact $X \times \{w \in \mathbb{C} : |w| \leq R\}$, we see that $Y$ is also compact. Let 
    \[ \epsilon := \min (\{ |P(x,w)| : (x,w) \in Y \} \cup \{1\}) \]
    which is positive as $Y$ does not contain a pair $(x, w)$ such that $P(x,w) = 0$. \par
    Now assume that $f \in C(X)$ is an $\epsilon$-approximate solution of $P$. Then by Lemma \ref{lem:polyRootBound} we get that $|f(x)| \leq 1 + \epsilon + \max_{0 \leq i < n} \norm{a_i} \leq R$ for all $x \in X$. However for each $x_0 \in X$, we have that $|P(x_0,f(x_0))| < \epsilon$, which by our choice of $\epsilon$, implies that $|P(x_0,f(x_0))| < |P(x,w)|$ for all $(x,w) \in Y$. But this means that for each $x_0 \in X$, the pair $(x_0, f(x_0))$ cannot be in $Y$, and since $|f(x_0)| \leq R$ this is only possible when $d(f(x_0), \sigma_{P(x_0,z)}) < \delta$.
\end{proof}

This allows us to prove our first stability result: for a polynomial with no repeated roots, having approximate roots is equivalent to having an exact root. 

\begin{thm}\label{thm:noRepeatRootsStability}
    Let $X$ be a compactum. If $P(x,z)$ is a monic polynomial with coefficients in $C(X)$ which has no repeated roots, then $P$ has approximate roots if and only if it has an exact root. 
\end{thm}
\begin{proof}
    Clearly if $P$ has an exact root then it has approximate roots. Conversely assume that $P$ has approximate roots. Let $\delta : X \to [0,\infty)$ be the function which assigns to each $x \in X$, the minimum distance between any two roots of of the complex polynomial $P(x,z)$, i.e. 
    \begin{align*}
        \delta(x) = \min\{ |\lambda - \mu| : \lambda, \mu \in \sigma_{P(x,z)}, \lambda \neq \mu \}
    \end{align*}
    Observe that $\delta$ is non-vanishing since $P$ has no repeated roots, and it is a continuous function by Lemma \ref{lem:continuousDependenceOfRoots}. In particular, it achieves a minimum since $X$ is compact; let $\alpha := \frac{1}{2}\min_{x \in X} \delta(x)$. By Lemma \ref{lem:approxSolutionsAreCloseToRoots}, we can find an $\epsilon > 0$ such that for all $\epsilon$-approximate solutions $f$ of $P$, we have that $d(f(x),\sigma_{P(x,z)}) < \frac{\alpha}{3}$ for all $x \in X$. \par
    Now since $P$ has approximate roots, we can find an $\epsilon$-approximate solution $f \in C(X)$ of $P$; by above we have that $d(f(x), \sigma_{P(x,z)}) < \frac{\alpha}{3}$ for all $x \in X$. This means that for each $x \in X$, there exists some $\lambda_x \in \sigma_{P(x,z)}$ such that $|f(x) - \lambda_x| < \frac{\alpha}{3}$; in fact, this $\lambda_x$ is unique, since if $\mu_x \in \sigma_{P(x,z)}$ is any other root of $P(x,z)$, then 
    \begin{align*}
        |f(x) - \mu_x| \geq |\lambda_x - \mu_x| - |f(x) - \lambda_x| \geq \delta(x) - \frac{\alpha}{3} \geq 2\alpha - \frac{\alpha}{3} \geq \frac{5}{3}\alpha \tag{$\clubsuit$}
    \end{align*}
    Thus we obtain a well-defined function $g : X \to \mathbb{C}$ given by $g(x) = \lambda_x$. By definition we have that $P(x,g(x)) = 0$ for all $x \in X$, so to show that $P$ has an exact root, it suffices to show that $g$ is continuous. To this end, take any $x_0 \in X$ and assume that $\epsilon' > 0$ is given. Since $f$ is continuous, and by Corollary \ref{cor:continuityOfRootMultiset}, the function $X \to \mathbb{C}^n/S_n , x  \mapsto \sigma_{P(x,z)}$ is continuous and $g(x_0) \in \sigma_{P(x_0,z)}$, we can find an open subset $U$ of $x_0$ such that whenever $x \in U$, we have that $d(g(x_0), \sigma_{P(x,z)}) < \min(\frac{\alpha}{3}, \epsilon')$ and $|f(x) - f(x_0)| < \frac{\alpha}{3}$. But this implies that for all $x \in U$, there is some $\lambda \in \sigma_{P(x,z)}$ such that $|\lambda - g(x_0)| < \min(\frac{\alpha}{3}, \epsilon')$, which implies that 
    \begin{align*}
        |\lambda - f(x)| &\leq |\lambda - g(x_0)| + |g(x_0) - f(x_0)| + |f(x_0) - f(x)| \\
        &\leq \frac{\alpha}{3} + \frac{\alpha}{3} + \frac{\alpha}{3} = \alpha
    \end{align*}
    The calculation $(\clubsuit)$ above shows that if $\mu$ is any root of $P(x,z)$ which is not $g(x)$, then it must be that $|f(x) - \mu| \geq \frac{5}{3}\alpha$. But since $|\lambda - f(x)| \leq \alpha$ and $\lambda \in \sigma_{P(x,z)}$, this means that $\lambda = g(x)$. Thus $|g(x) - g(x_0)| = |\lambda - g(x_0)| < \epsilon'$. As this holds for all $x \in U$, this proves that $g$ is continuous, as desired. 
\end{proof}

In particular, the notion of having approximate roots is interesting only for polynomials which can have repeated roots. Instead of placing a condition on the kind of polynomials we consider, we can instead ask the question of what kind of spaces $X$ have the property that every monic polynomial over $C(X)$ having approximate roots actually has exact roots. Our second stability result is the following theorem, which shows that the main topological obstruction to passing from approximate solutions to exact solutions is local-connectedness. This is a generalization of \cite{Miura1999}, where it is shown that if $X$ is locally connected, then $C(X)$ is square-root closed if and only if it has approximate square roots. 

\begin{thm}\label{thm:localConnectednessGivesExactFromApprox}
    Let $X$ be a locally connected compactum. If $P(x,z)$ is a monic polynomial with coefficients in $C(X)$, then $P$ has approximate roots if and only if it has an exact root.
\end{thm}
\begin{proof}
    Obviously if $P(x,z)$ has an exact root, then it has approximate roots. Conversely assume that $P(x,z)$ has approximate roots. Then for each integer $k \geq 1$, we can find a $\frac{1}{k}$-approximate root $f_k \in C(X)$ of $P(x,z)$. If we can show that $\{f_k\}_{k = 1}^\infty$ has a convergent subsequence $\{f_{k_{l}}\}_{l=1}^{\infty}$ converging to some $f_\infty \in C(X)$, then passing to this subsequence we get
    \[|P(x, f_\infty(x))| = \lim_{l \to \infty} |P(x, f_{k_l}(x))| \leq \lim_{k \to \infty} \frac{1}{k} = 0\]
    for each $x \in X$, showing that $f_\infty$ is an exact solution. To this end, we will show that $\{f_k\}_{k = 1}^\infty$ has a convergent subsequence by checking that this sequence is uniformly bounded and equicontinuous, and then applying the Arzelà–Ascoli Theorem. \par
    By Lemma \ref{lem:polyRootBound}, we have that $\norm{f_k}_\infty \leq 2 + M$ (where $M = \max\{\norm{a_0}, \ldots, \norm{a_{n-1}}\}$) for all $k$, showing the sequence is uniformly bounded. 
    For equicontinuity, pick a point $x_0 \in X$ and $\epsilon > 0$. Then we may factor our polynomial at this point
    \[P(x_0, z) = \prod_{i = 1}^r (z - \lambda_i)^{m_i}\]
    for some distinct roots $\lambda_i \in \mathbb{C}$ with respective multiplicities $m_i$. Let $\epsilon_0 := \frac{1}{2}\min(\epsilon, \min_{i \neq j} |\lambda_i - \lambda_j|)$. By Lemma \ref{lem:continuousDependenceOfRoots} we can find a neighbourhood $U_1 \subset X$ of $x_0$ such that 
    \begin{align*}
        \bigcup_{x \in U_1} \{ w \in \mathbb{C} : P(x,w) = 0 \} \subset \bigcup_{i=1}^{r} B_{\frac{\epsilon_0}{4}}(\lambda_i) 
    \end{align*}
    and since $X$ is locally connected, we can find a connected open neighbourhood $U$ of $x_0$ with $U \subset U_1$. Now by Lemma \ref{lem:approxSolutionsAreCloseToRoots}, we can find a positive integer $k_0$ such that $d(f_k(x), \sigma_{P(x,z)}) < \frac{\epsilon_0}{4}$ for all $k \geq k_0$ and all $x \in X$. Hence if $x \in U$ and $k \geq k_0$, then we can find a root $w_{k,x}$ of $P(x,z)$ such that $|f_k(x) - w_{k,x}| < \frac{\epsilon_0}{4}$, and an $i_{k,x} \in \{1, \ldots, r\}$ such that $w_{k,x} \in B_{\frac{\epsilon_0}{4}}(\lambda_{i_{k,x}})$. The triangle inequality immediately implies $|f_k(x) - \lambda_{i_{k,x}}| < \frac{\epsilon_0}{4} + \frac{\epsilon_0}{4} = \frac{\epsilon_0}{2}$ for all $ k\geq k_0$ and all $x \in U$. It follows that $f_k(x) \in \bigcup_{i = 1}^r B_{\frac{\epsilon_0}{2}}(\lambda_i)$ for all $x \in U$ and $k \geq k_0$. However, since $\epsilon_0 < \min_{i \neq j} |\lambda_i - \lambda_j|$, it follows that the family of disks $\{B_{\frac{\epsilon_0}{2}}(\lambda_i)\}_{i=1}^{r}$ are disjoint, so by connectedness of $U$ it follows that for each $k \geq k_0$, $f_k(U)$ must lie in a single such disk; that is, for each $k \geq k_0$ we can find some $i_k$ such that $f_k(U) \subset B_{\frac{\epsilon_0}{2}}(\lambda_{i_k})$. Therefore for all $x \in U$ and all $k \geq k_0$ we have that:
    \begin{align*}
        |f_k(x) - f_k(x_0)| \leq |f_k(x) - \lambda_{i_k}| + |\lambda_{i_k} - f_k(x_0)| < \frac{\epsilon_0}{2} + \frac{\epsilon_0}{2} < \epsilon
    \end{align*}
    Therefore the sequence $\{f_k\}_{k=k_0}^{\infty}$ is equicontinuous, and since $f_1, \ldots, f_{k_0}$ are continuous functions, it follows that the sequence $\{f_k\}_{k=1}^{\infty}$ is equicontinuous as desired. 
\end{proof}

\begin{cor}\label{cor:local_conn_approxAC_is_AC}
    If $X$ is a locally connected compactum, then $C(X)$ is approximately algebraically closed if and only if it is algebraically closed. 
\end{cor}

We now turn our attention to perturbation, i.e. describing a situation in which we can perturb a polynomial so that it has no repeated roots. As noted in Section \ref{subsec:prelim_spaces_of_polys}, the space of polynomials with no repeated roots is identified with the subset $C(X, B_n)$ of the C*-algebra $C(X,\mathbb{C}^n)$, so precisely the perturbation question means describing when $C(X,B_n)$ is dense in $C(X,\mathbb{C}^n)$. \par

\begin{thm}\label{thm:noRepeatRootsDense}
    Let $X$ be a compactum with $\dim X \leq 1$, and let $P:X \to \bC^n$ be a monic polynomial over $C(X)$ with degree $n$. Then for all $\epsilon > 0$, there exists a monic polynomial over $C(X)$ of degree $n$ with no repeated roots $Q: X \to B_n$ such that $\norm{P-Q} < \epsilon$.
\end{thm}
\begin{proof}
    We need to show that $C(X, B_n)$ is dense in $C(X, \bC^n)$. Note that $C(X,B_n)$ is precisely the set of functions in $C(X,\mathbb{C}^n)$ whose image does not intersect the monic discriminant variety $V(\Delta)$.

    First, observe that $V(\Delta)$ is the zero-set of a complex polynomial in $n$ variables, and hence it has real codimension $2$ and is a closed subset of $\mathbb{C}^n$.  By \cite{discriminantStratification}, the variety $V(\Delta)$ admits a Whitney stratification $D = D_2 \cup D_3 \cdots \cup D_n$ where each $D_j$ consists of those points $(a_0, \ldots, a_{n-1}) \in V(\Delta)$ where the highest multiplicity root of the corresponding polynomial $a_0 + a_1 z + \cdots + a_{n-1}z^{n-1} + z^n$ is $j$. Each $D_j$ is a real smooth submanifold of $\mathbb{C}^n$ which is a closed subset $V(\Delta)$ and has codimension at least that of $V(\Delta)$; hence each has real codimension at least $2$.

    Let $f : X \to \mathbb{C}^n$ be a continuous function, and let $\epsilon > 0$. Since $\dim X \leq 1$, $X$, it follows from \cite[Corollary 1]{mardesic} that $X$ is a double inverse limit of a doubly-indexed family of polyhedra, each of which have dimension at most one. Hence there exists a compact polyhedron $K$ with $\dim K \leq 1$, a map $q : X \to K$, and a continuous function $g : K \to \mathbb{C}^n$ such that $\norm{f - g \circ q} < \frac{\epsilon}{2}$. Since $D_2$ is a smooth submanifold of $\mathbb{C}^n$ with real codimension at least 2, we have
    \begin{align*}
        \dim_{\mathbb{R}} D_2 +
        \dim_{\mathbb{R}} K \leq 2n - 1 < 2n = \dim_{\mathbb{R}} \mathbb{C}^n,
    \end{align*}
    so a general position argument allows us to perturb $g$ to find some continuous function $g_2 : K \to \mathbb{C}^n \setminus D_2$ with $\norm{g - \iota_2 \circ g_2} < \frac{\epsilon}{2n}$ where $\iota_2 : \mathbb{C}^n \setminus D_2 \hookrightarrow \mathbb{C}^n$ is the open inclusion. Then since $\mathbb{C}^n \setminus D_2$ is an open subset of $\mathbb{C}^n$, it is in particular a smooth manifold with
    \begin{align*}
        \dim_{\mathbb{R}} D_3 + \dim_{\mathbb{R}} K \leq 2n - 1 < 2n = \dim_{\mathbb{R}} (\mathbb{C}^n \setminus D_2),
    \end{align*}
    so a general position argument allows us to perturb $g_2$ to find a continuous function $g_3 : K \to (\mathbb{C}^n \setminus D_2) \setminus D_3$ with $\norm{g_2 - \iota_{3} \circ g_3} < \frac{\epsilon}{2n}$ where $\iota_3 : (\mathbb{C}^n \setminus D_2) \setminus D_3 \hookrightarrow \mathbb{C}^n \setminus D_2$ is the open inclusion. Continuing in this way, we obtain functions $g_2, g_3, \ldots, g_n$ with $g_j : K \to \mathbb{C}^n \setminus(D_2 \cup \cdots \cup D_j)$ satisfying $\norm{g_j - \iota_{j+1} \circ g_{j+1}} < \frac{\epsilon}{2n}$ where $\iota_{j+1} : \mathbb{C}^n \setminus (D_2 \cup \cdots \cup D_{j+1}) \hookrightarrow \mathbb{C}^n \setminus (D_2 \cup \cdots \cup D_{j})$ is the open inclusion for each $j =2 ,3 \cdots, n-1$. Finally, setting $h : K \to \mathbb{C}^n$ to be the continuous function $h 
    := \iota_2 \circ \iota_3 \circ \cdots \circ \iota_n \circ g_n$, we have that the image of $h$ does not intersect $V(\Delta)$, and $\norm{h - g} \leq \frac{\epsilon}{2}$. Hence $\norm{h - f} < \epsilon$. 
\end{proof}

If we want to use this perturbation result to say something about approximate roots, we should also check that by perturbing the coefficients of a polynomial over $C(X)$ does not change its approximate roots. 

\begin{lem}\label{lem:closePolyMeansCloseRoots}
    Let $X$ be a compactum, and let $P(x,z) = z^n + a_{n-1}(x)z^{n-1} + \cdots + a_1(x)z + a_0(x)$ be a monic polynomial with coefficients in $C(X)$. Then there exists a constant $C > 0$ such that if $\epsilon \in (0,1]$ and if $Q(x,z) = z^n + b_{n-1}(x)z^{n-1} + \cdots + b_1(x)z + b_0(x)$ is a monic polynomial with coefficients in $C(X)$ such that $\norm{a_k - b_k}_{\infty} < \epsilon$, then every $\epsilon$-approximate solution of $Q$ is a $C\epsilon$-approximate solution of $P$. 
\end{lem}
\begin{proof}
    Let $M = \max\{ \norm{a_0}, \norm{a_1}, \ldots, \norm{a_{n-1}} \}$ and let $C =  \frac{(2+M)^n - 1}{1+M} + 1$. Take a monic polynomial $Q(x,z) = z^n + b_{n-1}(x)z^{n-1} + \cdots + b_1(x)z + b_0(x)$ with coefficients in $C(X)$ such that $\norm{a_k - b_k}_{\infty} < \epsilon$. Let $f \in C(X)$ be an $\epsilon$-approximate solution of $Q$. Then by Lemma \ref{lem:polyRootBound}, $\norm{f} \leq 1 + \epsilon + M \leq 2 + M$, so for each $x \in X$, we have that
    \begin{align*}
		|P(x,f(x))| &\leq |P(x,f(x)) - Q(x,f(x))| + |Q(x,f(x))| \\
			&=|(a_{n-1}(x) - b_{n-1}(x))f(x)^{n-1} + \cdots + (a_1(x)-b_1(x))f(x) + (a_0(x)- b_0(x))| \\
                & \qquad + |Q(x,f(x))| \\
			&\leq \norm{a_{n-1} - b_{n-1}}\norm{f}^{n-1} + \cdots + \norm{a_1 - b_1}\norm{f} + \norm{a_0 - b_0} + \epsilon \\
			&< \epsilon(\norm{f}^{n-1} + \cdots + \norm{f} + 1) + \epsilon \\
			&\leq \epsilon\left( (2+M)^{n-1} + \cdots + (2+M) + 1 \right) + \epsilon \\
			&= \epsilon \left( \frac{(2+M)^n - 1}{(2+M) - 1} + 1 \right) \\
			&= C\epsilon
		\end{align*}
    Since the above holds for all $x \in X$, it follows that $\sup_{x \in X} |P(x,f(x))| < C\epsilon$, as desired.
\end{proof}

We are now ready to prove the main result of this section, Theorem \ref{thm:perturbAndStabDegN}.

\begin{proof}[Proof of Theorem \ref{thm:perturbAndStabDegN}]
    Clearly (1) implies (2), and (2) implies (3) by Theorem \ref{thm:noRepeatRootsStability}. Now assume that (3) holds, and take any monic polynomial $P$ of degree at most $n$ with coefficients in $C(X)$. By Lemma \ref{lem:closePolyMeansCloseRoots}, we can find a constant $C > 0$ such that if $\epsilon_0 \in (0,1)$ and if $Q$ is a monic polynomial with coefficients in $C(X)$ such that $\norm{P - Q} < \epsilon_0$, then every $\epsilon_0$-approximate solution of $Q$ is a $C\epsilon_0$-approximate solution of $P$. Now suppose that $\epsilon > 0$, and let $\epsilon_0 := \min(1, \frac{\epsilon}{C})$. Since $\dim X \leq 1$, by Theorem \ref{thm:noRepeatRootsDense}, we can find a monic polynomial $Q$ with coefficients in $C(X)$ having no repeated roots such that $\deg Q = \deg P \leq n$ and $\norm{P - Q} < \epsilon_0$. By assumption, $Q$ has an exact root, say $f \in C(X)$. Then in particular $f$ is an $\epsilon_0$-approximate root of $Q$, so it follows by our choice of $C$ that $f$ is a $C\epsilon_0$-approximate solution of $P$. But $\epsilon \geq C \epsilon_0$, so it follows that $f$ is an $\epsilon$-approximate solution of $P$. This proves that (1) holds, as desired. 
\end{proof}

Now Lemma 2.4 in \cite{KawamuraMiura2009} and the discussion thereafter show that if every monic polynomial over $C(X)$ of degree at most $n$ with no repeated roots has an exact root, then any monic polynomial $P$ over $C(X)$ with degree at most $n$ can be factored completely, that is, we can find continuous functions $\lambda_1, \ldots, \lambda_n \in C(X)$ such that $P(x,z) = \prod_{j=1}^{n}(z-\lambda_j(x))$. Thus as a consequence of Theorem \ref{thm:perturbAndStabDegN} we obtain:
\begin{cor}\label{cor:completelySolvableDegN}
    Let $X$ be a compactum with $\dim X \leq 1$ and let $n$ be a positive integer. The following are equivalent: 
    \begin{enumerate}
        \item Every monic polynomial of degree at most $n$ with coefficients in $C(X)$ has approximate roots;
        \item Every monic polynomial of degree at most $n$ with coefficients in $C(X)$ with no repeated roots has approximate roots;
        \item Every monic polynomial of degree at most $n$ with coefficients in $C(X)$ with no repeated roots can be factored completely.
    \end{enumerate}
\end{cor}
In the language of \cite{GorinLin}, the last condition in the corollary above says that the class of equations $\overline{\mathfrak{A}_n}(X)$ consisting of monic polynomials over $C(X)$ of degree at most $n$ with no repeated roots is \emph{completely solvable}.

\section{Shape theory}\label{sec:proCatShape}

\subsection{Pro-categories}

To generalize the global arguments of homotopy theory to spaces that do not have nice local properties (such as local path-connectedness), we need \emph{shape theory}. In this section, following \cite{dydak_segal}, we recall some elements of shape theory and the language in which it is formulated, namely that of pro-categories.

\begin{defn}\label{defn:proCat}
    Given a category $\cC$, we define its corresponding \emph{pro-category} $\Pro(\cC)$ to be the category consisting of the following objects, morphisms, and composition law:
    \begin{enumerate}[(i)]
        \item The objects of $\Pro(\cC)$ are given by inverse systems $(X_\alpha, p_{\alpha}^{\alpha'}, A)$, which are functors from a directed set $A$ to $\cC$. Explicitly, for each $\alpha \in A$ we have an object $X_\alpha$ of $\cC$, and for each pair $\alpha' \geq \alpha$ we have a morphism $p_{\alpha}^{\alpha'}: X_{\alpha'} \to X_\alpha$.
        \item The set of morphisms in $\Pro(\cC)$ between two objects $\ul{X} = (X_\alpha, p_{\alpha}^{\alpha'}, A)$ and $\ul{Y} = (Y_\beta, q_{\beta}^{\beta'}, B)$ is given by
        \[\Hom_{\Pro(\cC)}(\ul{X}, \ul{Y}) = \varprojlim_{\beta \in B} \varinjlim_{\alpha \in A} \Hom_{\cC}(X_\alpha, Y_\beta)\]
    \end{enumerate}

    Before defining the composition of general morphisms in $\Pro(\cC)$, it is worth noting a few remarks concerning the above definitions. To start, we have an inclusion of objects $\iota: \cC \to \Pro(\cC)$ by using indexing sets of a single element
    \[X \mapsto (X, \id_X, *)\]
    and we'd like this to be a fully faithful inclusion (the only thing left to check is that $\iota$ preserves composition of morphisms, and hence is a functor). Under this inclusion functor $\iota$ the image of the inverse system $(X_\alpha, p_{\alpha}^{\alpha'}, A)$ in $\cC$ is an inverse system $(\iota(X_\alpha), \iota(p_{\alpha}^{\alpha'}), A)$ in $\Pro(\cC)$, and the main idea in the definition of $\Pro(\cC)$ is that this inverse system $(\iota(X_\alpha), \iota(p_{\alpha}^{\alpha'}), A)$ has a natural projective limit in $\Pro(\cC)$ given precisely by the object $\ul{X} = (X_\alpha, p_{\alpha}^{\alpha'}, A)$. This perspective lets us understand morphisms into an object $\ul{Y} = (Y_\beta, q_{\beta}^{\beta'}, B)$ of $\Pro(\cC)$ by understanding a system of morphisms into objects $\iota(Y_\beta)$ of $\iota(\cC) \cong \cC$:
    \begin{align*}
        \Hom_{\Pro(\cC)}(\ul{X}, \ul{Y}) & = \Hom_{\Pro(\cC)}(\ul{X}, \varprojlim_{\beta \in B} \iota(Y_\beta)) \\
        & = \varprojlim_{\beta \in B} \Hom_{\Pro(\cC)}(\ul{X}, \iota(Y_\beta))
    \end{align*}
    using the assertion the composition law will be such that $\ul{Y}$ is a projective limit in $\Pro(\cC)$ for the inverse system $(\iota(Y_\beta), \iota(q_{\beta}^{\beta'}), B)$.
    
    Unlike for the second entry, we cannot in general pull the projective limit out of the first entry. Thus the following identity
    \[\Hom_{\Pro(\cC)}(\ul{X}, \iota(Y)) = \varinjlim_{\alpha \in A} \Hom_{\cC}(X_\alpha, Y)\]
    is really a definition. Note that an element $f$ of this direct limit is an equivalence class of morphisms, with representatives given by morphisms $f_\alpha: X_\alpha \to Y$ for some $\alpha \in A$; two representatives $f_\alpha \in \Hom_{\cC}(X_\alpha, Y)$ and $f_{\alpha'} \in \Hom_{\cC}(X_{\alpha'}, Y)$ represent the same class if there exists some $\alpha'' \in A$ with $\alpha'' \geq \alpha$, and $\alpha'' \geq \alpha'$, and
    \[f_\alpha \circ p_{\alpha}^{\alpha''} = f_{\alpha'} \circ p_{\alpha'}^{\alpha''}\]
    With this background we are ready for the definition of the composition law in $\Pro(\cC)$:
    \begin{enumerate}[(i)]
        \item[(iii)] Given three objects $\ul{X} = (X_\alpha, p_{\alpha}^{\alpha'}, A)$, $\ul{Y} = (Y_\beta, q_{\beta}^{\beta'}, B)$, and $\ul{Z} = (Z_\gamma, r_{\gamma}^{\gamma'}, C)$ the composition of morphisms
        \begin{align*}
        \Hom_{\Pro(\cC)}(\ul{X}, \ul{Y}) \times \Hom_{\Pro(\cC)}(\ul{Y}, \ul{Z}) & \longrightarrow \Hom_{\Pro(\cC)}(\ul{X}, \ul{Z}) \\
        (f, g) & \longmapsto g \circ f
        \end{align*}
        is described by giving a representative of $(g \circ f)_\gamma : \ul{X} \to Z_\gamma$ for each $\gamma \in C$. We start by picking a representative $g_\gamma^\beta : Y_\beta \to Z_\gamma$ for $g_\gamma : \ul{Y} \to Z_\gamma$, along with a representative $f_\beta^\alpha : X_\alpha \to Y_\beta$ for $f_\beta : \ul{X} \to Y_\beta$. Then a representative for $(g \circ f)_\gamma : \ul{X} \to Z_\gamma$ is given by $g_\gamma^\beta \circ f_\beta^\alpha: X_\alpha \to Z_\gamma$. It is an exercise to show that this definition of $(g \circ f)_\gamma$ is independent of the choices of representatives made for $g_\gamma$ and $f_\beta$, and that the $(g \circ f)_\gamma$ respect the structure maps $r_{\gamma}^{\gamma'}$ of $\ul{Z}$.
    \end{enumerate}
\end{defn}

Henceforth we will implicitly use that $\cC$ includes fully faithfully into $\Pro(\cC)$, and stop using $\iota$ to explicitly notate this inclusion functor.

\begin{defn}\label{rem:proFunctor}
We have a functor $\Pro$ from categories to categories which takes a category $\cC$ to its pro-category $\Pro(\cC)$, and takes a functor $\cF: \cC \to \cD$ to the functor $\Pro(\cF)$ that applies $\cF$ to all objects (and structure maps of the inverse systems), and also applies $\cF$ to morphisms (by doing so on representatives). In particular, a subcategory $\cD \seq \cC$ gives rise to a subcategory $\Pro(\cD) \seq \Pro(\cC)$.
\end{defn}

While part of the definition of morphisms in $\Pro(\cC)$ requires inductive limits of morphisms in $\cC$, and hence require choices of representatives for certain definitions (such as the composition law) and proofs, our aim should be to state our results without making these choices. In light of this, our definitions and results should be up to isomorphism in $\Pro(\cC)$, so let us take the time now to understand what this means. \par

First note that for each object $\ul{X} = (X_\alpha, p_{\alpha}^{\alpha'}, A)$ of $\Pro(\cC)$ we may define a map $p_\alpha : \ul{X} \to X_\alpha$ by using the identity map $\id_{X_\alpha} = p_\alpha^\alpha: X_\alpha \to X_\alpha$ as a representative (in which case, for all $\alpha' \geq \alpha$ the structure maps $p_\alpha^{\alpha'}$ are also representatives). Then this system of morphisms $p_\alpha : \ul{X} \to X_\alpha$ assembles into a map $\id_{\ul{X}} : \ul{X} \to \ul{X}$ which we immediately see is an identity for $\ul{X}$ using the above composition law. \par
The following proposition gives us a characterization of (left and right) inverses, which we'll use later to prove some properties are invariant under isomorphisms in $\Pro(\cC)$. The proposition follows immediately from this characterization of $\id_{\ul{X}}$ and the composition law.

\begin{prop}\label{prop:inverseProcategory}
   Consider two objects $\ul{X} = (X_\alpha, p_{\alpha}^{\alpha'}, A)$ and $\ul{Y} = (Y_\beta, q_{\beta}^{\beta'}, B)$ of $\Pro(\cC)$, along with morphisms $\ul{f}: \ul{X} \to \ul{Y}$ and $\ul{g}: \ul{Y} \to \ul{X}$. Then $\ul{g} \circ \ul{f} = \id_{\ul{X}}$ if and only if for any $\alpha \in A$ and any $\beta \in B$ such that there exists a representative $g_\alpha^\beta : Y_\beta \to X_\alpha$ of $g_\alpha : \ul{Y} \to X_\alpha$, there exists an $\alpha' \in A$ such that there exists a representative $f_\beta^{\alpha'} : X_{\alpha'} \to Y_\beta$ for $f_\beta : \ul{X} \to Y_\beta$ such that $g_\alpha^\beta \circ f_\beta^{\alpha'} = p_\alpha^{\alpha'}$.
\end{prop}

\subsection{Pro-groups}\label{subsec:pro-groups}

One of the main categories to which we'll apply this $\Pro$ functor is the category of groups $\Grp$. In this section we will give some examples and properties of pro-groups that we will use later on. 

\begin{examples}\label{ex:nestedSubgroupsInverseSystem}
Given a sequence $G_1 \supseteq G_2\supseteq G_3 \supseteq \cdots$ of nested subgroups there is a pro-group $\ul{G}$ given by the inverse system of inclusions
\[\begin{tikzcd}
	G_1 & G_2 & G_3 & \cdots
	\arrow[from=1-2, to=1-1]
	\arrow[from=1-3, to=1-2]
	\arrow[from=1-4, to=1-3]
\end{tikzcd}\]
of subsequent subgroups. To be explicit, the indexing set here is $\bN_{\geq 1}$, the objects are the subgroups $G_k$ and the structure maps $p_k^{k'}$ are the inclusions $G_{k'} \to G_k$. A common example is to take a sequence of positive integers $\nu = (n_1, n_2, n_3, \dots)$ such that $n_k|n_{k + 1}$ for each $k$, and then consider the inverse system
\[\begin{tikzcd}
	n_1 \bZ & n_2 \bZ & n_3 \bZ & \cdots
	\arrow[from=1-2, to=1-1]
	\arrow[from=1-3, to=1-2]
	\arrow[from=1-4, to=1-3]
\end{tikzcd}\]
of subgroups of $\bZ$. We'll use $\ul{\nu\bZ}$ to denote the pro-group defined by this inverse system.
\end{examples}

This example brings us to one of the first properties we can define about pro-groups. For this, we will need to recall the following property of groups.

\begin{defn}\label{defn:divisibility_group}
    A group $G$ is called \emph{$m$-divisible} if for every $g \in G$ there exists some $h \in G$ such that $h^m = g$. We shall call $G$ \emph{divisible} if it is $m$-divisible for every integer $m \geq 1$.
\end{defn}

\begin{defn}\label{defn:divisibility_progroup}
    An inverse system of groups $(G_\alpha, p_\alpha^{\alpha'}, A)$ is called \emph{$m$-divisible} if for any $\alpha \in A$ there exists an $\alpha' \geq \alpha$ so that for any $g \in G_{\alpha'}$ there exists an element $x \in G_\alpha$ so that $p_\alpha^{\alpha'}(g) = x^m$. The following proposition shows being $m$-divisible is an isomorphism invariant for inverse systems of groups in $\Pro(\Grp)$. Thus we can call a pro-group $\ul{G}$ \emph{$m$-divisible} if one (and therefore all) inverse systems isomorphic to it are $m$-divisible.
\end{defn}

\begin{prop}\label{prop:divisibility_progroup_iso_invariant}
    Consider two inverse systems of groups $\ul{G} = (G_\alpha, p_\alpha^{\alpha'}, A)$ and $\ul{H} = (H_\beta, q_\beta^{\beta'}, B)$ along with a morphism $\ul{\phi}: \ul{H} \to \ul{G}$ in $\Pro(\Grp)$ which has a right inverse. If $\ul{H}$ is $m$-divisible, then $\ul{G}$ is also $m$-divisible.
\end{prop}
\begin{proof}
    By assumption there exists a morphism $\ul{\psi}: \ul{G} \to \ul{H}$ such that $\ul{\phi} \circ \ul{\psi} = \id_{\ul{G}}$. Now given some $\alpha \in A$, pick some $\beta \in B$ such that there exists a representative $\phi_\alpha^\beta: H_\beta \to G_\alpha$ for $\phi_\alpha: \ul{H} \to G_\alpha$. Since $\ul{H}$ is $m$-divisible, we can find some $\beta' \in B$ such that the image of elements under $q_\beta^{\beta'}$ are $m$th powers of elements in $H_\beta$. \par
    By Proposition \ref{prop:inverseProcategory}, since $\phi_\alpha^{\beta'} = \phi_\alpha^\beta \circ q_\beta^{\beta'}$ is a representative for $\phi_\alpha : \ul{H} \to G_\alpha$, we know that there exists some $\alpha' \in A$ along with a representative $\psi_{\beta'}^{\alpha'}: G_{\alpha'} \to H_{\beta'}$ for $\psi_{\beta'}: \ul{G} \to H_{\beta'}$ such that $\phi_\alpha^{\beta'} \circ \psi_{\beta'}^{\alpha'} = p_\alpha^{\alpha'}$. \par
    We are now ready to prove that the images of elements under $p_\alpha^{\alpha'}$ are $m^{\text{th}}$ powers of elements in $G_\alpha$. Given some $g \in G_{\alpha'}$ we know that there exists some $y \in H_\beta$ such that $q_\beta^{\beta'}(\psi_{\beta'}^{\alpha'}(g)) = y^m$. Then defining $x = \phi_\alpha^\beta(y) \in G_\alpha$ we find
    \[p_\alpha^{\alpha'}(g) = (\phi_\alpha^{\beta'} \circ \psi_{\beta'}^{\alpha'})(g) = (\phi_\alpha^\beta \circ q_\beta^{\beta'} \circ \psi_{\beta'}^{\alpha'})(g) = \phi_\alpha^\beta(y^m) = x^m\]
    Since $\alpha \in A$ and $g \in G_{\alpha'}$ were arbitrary, this shows $\ul{G}$ is $m$-divisible.
\end{proof}

Consider for example the trivial pro-group $\ul{1} = (1, \id, *)$, given by taking the trivial group $1$ in $\Grp$ and passing it through the inclusion functor $\Grp \to \Pro(\Grp)$. In this way we regard $\ul{1}$ as an inverse system indexed by a single element, consisting just of the trivial group $1$. From this description it is immediate that $\ul{1}$ is $m$-divisible for all positive integers $m$. By Proposition \ref{prop:divisibility_progroup_iso_invariant} any inverse system isomorphic to the trivial pro-group $\ul{1}$ in $\Pro(\Grp)$ is also $m$-divisible for all positive integers $m$. We can get examples of non-trivial pro-groups being $m$-divisible via the next proposition. 

\begin{prop}\label{prop:subgroupsZdivisibility}
    Consider a sequence of natural numbers $\nu = (n_1, n_2, n_3, \dots)$ such that $n_k|n_{k + 1}$ for each $k$. The pro-group $\ul{\nu\bZ}$ defined as in Example \ref{ex:nestedSubgroupsInverseSystem} is $m$-divisible if and only if for each $a \in \bN$ there exists some $j \in \bN_{\geq 1}$ such that $m^a|n_j$.
\end{prop}
\begin{proof}
    To start assume that for each $a \in \bN$ there exists some $j \in \bN_{\geq 1}$ such that $m^a|n_j$. To show that $\ul{\nu\bZ}$ is $m$-divisible we need start by picking some index $k \in \bN_{\geq 1}$. Let $b \in \bN$ be the exponent such that $n_k = qm^b$ with $q \in \bN$ relatively prime to $m$. By our assumption we may find some $j \in \bN_{\geq 1}$ such that $m^{b+1}|n_j$. Since $n_j$ has at least one more factor of $m$, it must be that $k \leq j$ and $n_k|n_j$ (as otherwise $j < k$ and $n_j|n_k$ which is impossible). Since $n_k|n_j$ we get $q|n_j$, and along with $m^{b+1}|n_j$ we get that $qm^{m + 1}|n_j$ as they are relatively prime. In particular
    \[n_j\bZ \seq qm^{m + 1} \bZ = m n_k \bZ\]
    which tells us that every element of $n_j\bZ$ is an $m$th multiple of an element in $n_k \bZ$. Since $p_k^j: n_j\bZ \to n_k\bZ$ is just the inclusion map, this is exactly what we needed to show. \par
    
    Now assume that $\ul{\nu\bZ}$ is $m$-divisible. We prove the statement
    \begin{center}
    ``For each $a \in \bN$ there exists some $j \in \bN_{\geq 1}$ such that $m^a|n_j$.''
    \end{center}
    by induction on $a$. For the base case $a = 0$ we can pick $j = 1$ (or any other $j$). For the inductive step assume that the statement holds for $a$, meaning that there is some $k \in \bN_{\geq 1}$ such that $m^a|n_k$. By assumption of $m$-divisibility of $\ul{\nu\bZ}$ there exists some $j \geq k$ such that the image of elements in $n_j\bZ$ under $p_k^j: n_j\bZ \to n_k\bZ$ are $m$th multiples. In particular picking $n_j \in n_j\bZ$ there is an element $x \in n_k\bZ$ such that $n_j = p_k^j(n_j) = m \cdot x$ (note that $p_k^j$ is just inclusion). Since $x \in n_k\bZ$ and $m^a|n_k$ we have that $m^a|x$ and hence $n_j = m \cdot x$ is divisible by $m^{a+1}$. This completes the inductive step.
\end{proof}

\begin{examples}\label{ex:divisibleProgroups}
Consider $\nu = (2, 2^2, 2^3, \dots)$ the sequence of integers $n_k = 2^k$ and the associated pro-group $\ul{\nu\bZ}$. By Proposition \ref{prop:subgroupsZdivisibility} we see that $\ul{\nu\bZ}$ is $2$-divisible but it is not $p$-divisible for any other prime number $p$. In particular $\ul{\nu\bZ}$ not being $m$-divisible for some $m$ means $\ul{\nu\bZ}$ is not isomorphic to the trivial pro-group.
\end{examples}

This example highlights that we can construct pro-groups with certain properties (for example $2$-divisibility) without them being the trivial group. However, if instead of $\nu = (2, 2^2, 2^3, \dots)$ we used some $\nu = (n_1, n_2, n_3, \dots)$ whose terms were for each prime $p$ eventually arbitrarily $p$-divisible, we would not be able to argue that $\ul{\nu\bZ}$ is non-trivial in the same way. To answer this question and others down the line we can apply the following proposition.

\begin{prop}\label{prop:trivial_progroup_nested_subgroups}
    Let $\ul{G} = (G_\alpha, p_\alpha^{\alpha'}, A)$ be an inverse system of groups. Then $\ul{G}$ is isomorphic to the trivial pro-group $\ul{1}$ if and only if for each $\alpha \in A$ there exists some $\alpha' \geq \alpha$ such that $p_\alpha^{\alpha'}$ is the trivial morphism. \par
    In particular if $\ul{G}$ is an inverse system indexed by $\bN_{\geq 1}$ given by inclusions of nested subgroups $G_k \supseteq G_{k + 1}$ as in Example \ref{ex:nestedSubgroupsInverseSystem}, then $\ul{G}$ is isomorphic to $\ul{1}$ if and only if the subgroups $G_k$ are eventually the trivial subgroup.
\end{prop}
\begin{proof}
    Assume that $\ul{G} \cong \ul{1}$, and let $\alpha  \in A$. Then we can find morphisms $\ul{f} : \ul{G} \to \ul{1}$ and $g : \ul{1} \to \ul{G}$ such that $\ul{g} \circ \ul{f} = \id_{\ul{G}}$, and by Proposition \ref{prop:inverseProcategory}, this means that there exists some $\alpha' \in A$ with $\alpha' \geq \alpha$ and maps $f^{\alpha'} : G_{\alpha'} \to 1$ and $g_{\alpha} : 1 \to G_{\alpha}$ such that $g_{\alpha} \circ f^{\alpha'} =p_{\alpha}^{\alpha'}$. But $g_{\alpha} \circ f^{\alpha'} : G_{\alpha'} \to G_{\alpha}$ must be the trivial map, since it factors through the trivial group. Hence $p_{\alpha}^{\alpha'}$ is the trivial map. \par
    Conversely, assume that for each $\alpha \in A$ there exists some $\alpha' \geq \alpha$ such that $p_{\alpha}^{\alpha'}$ is the trivial map. Then the collection of trivial morphisms $f^{\alpha} : G_{\alpha} \to 1$ define a morphism $\ul{f} : \ul{G} \to \ul{1}$ and similarly the trivial inclusions $g_{\alpha} : 1 \to G_{\alpha}$ assemble to define a morphism $\ul{g} : \ul{1} \to \ul{G}$. Obviously $\ul{f} \circ \ul{g}: \ul{1} \to \ul{1}$ is the identity, so we just need to check $\ul{g} \circ \ul{f} = \id_{\ul{G}}$. To do this, we can just use Proposition \ref{prop:inverseProcategory}. To this end, let $\alpha \in A$, and let $g_{\alpha} : 1 \to G_{\alpha}$. By our assumption, there is some $\alpha' \geq \alpha$ such that $p_{\alpha}^{\alpha'}$ is the trivial map, and hence $g_{\alpha} \circ f^{\alpha'} = p_{\alpha}^{\alpha'}$, since $g_{\alpha} \circ f^{\alpha'}$ factors through the trivial group. Hence by the criterion from Proposition \ref{prop:inverseProcategory}, $\ul{G} \cong \ul{1}$. \par
    The second part of the Proposition is an immediate application of the first part. 
\end{proof}

Typically, abelianization of a group is viewed as functor from the category of groups to the category of abelian groups. In order for us to stay within the category of groups and pro-groups, we will as a convention always post-compose this functor with the inclusion of abelian groups into groups. Thus for us, abelianization for groups is an endofunctor $\Ab : \Grp \to \Grp$. Simply by applying the functor $\Pro$, we obtain an endofunctor $\Pro(\Ab) : \Pro(\Grp) \to \Pro(\Grp)$, and we can make the following definition. 

\begin{defn}\label{defn:abelianization_progroup}
    Let $\ul{G}$ be a pro-group. The \emph{abelianization} of $\ul{G}$ is the pro-group $\Pro(\Ab)(\ul{G})$. We say that $\ul{G}$ is \emph{abelian} if $\Pro(\Ab)(\ul{G})$ is isomorphic to $\ul{G}$. 
\end{defn}

Later, we will want to understand when a pro-group is abelian. This leads to the following expected characterization of abelian pro-groups. 

\begin{prop}\label{prop:abelian_pro_groups-abelian_inverse_systems}
    Let $\ul{G} = (G_\alpha, p_\alpha^{\alpha'}, A)$ be a pro-group. Then $\ul{G}$ is abelian if and only if it is isomorphic to an inverse system of abelian groups. 
\end{prop}
\begin{proof}
    If $\ul{G}$ is abelian, then it is by definition isomorphic to $\Pro(\Ab)(\ul{G})$, which is an inverse system of abelian groups. On the other hand, assume that $\ul{G}$ is isomorphic to an inverse system $\ul{H} = (H_{\beta}, q_{\beta}^{\beta'}, B)$ where each $H_{\beta}$ is an abelian group. Then since each $H_{\beta}$ is an abelian group, each $H_{\beta}$ is naturally isomorphic to its abelianization $\Ab(H_{\beta})$. This implies that $\Pro(\Ab)(\ul{H})$ is isomorphic to $\ul{H}$, and hence we have that
    \begin{align*}
        \Pro(\Ab)(\ul{G}) \cong \Pro(\Ab)(\ul{H}) \cong \ul{H} \cong \ul{G}
    \end{align*}
    so $\ul{G}$ is abelian. 
\end{proof}

\subsection{Shape theory}

The goal of shape theory is to understand homotopy classes of maps from a general topological space into a CW complex. Let us denote the category of topological spaces with continuous maps as $\Spc$ and the full subcategory of CW complexes by $\CW \seq \Spc$. If we are only interested in maps up to homotopy, we can consider the category $\Ho\Spc$ whose objects are topological spaces and morphisms are continuous maps up to homotopy, and the analogous full subcategory $\Ho\CW \seq \Ho\Spc$ of CW complexes (and morphisms are still continuous maps up to homotopy equivalence). \par
Analogously we can define $\Spc_*$ and $\CW_* \seq \Spc_*$ to be the categories of pointed topological spaces and pointed CW complexes, both with morphisms given by basepoint preserving continuous maps. Then $\Ho\Spc_*$ and $\Ho\CW_* \seq \Ho\Spc_*$ are taken to have morphisms given by basepoint preserving continuous maps up to homotopies that fix the basepoint. As noted in 3.1.2 of \cite{dydak_segal} the treatment of pointed and unpointed spaces are analogous. Definitions can be made in the pointed case and after surpressing the basepoint we get definitions in the unpointed case. We shall follow this convention by making our definitions in the pointed case.

\begin{defn}\label{defn:shapeOfSpace}
    Let $(X, x_0)$ be a pointed topological space. Suppose we have an object $(\ul{W}, \ul{w}) = (W_\alpha, w_\alpha, p_{\alpha}^{\alpha'}, A) \in \Pro(\Ho\CW_*)$ along with a morphism $\ul{q}: (X, x_0) \to (\ul{W}, \ul{w})$ in $\Pro(\Ho\Spc_*)$, which is just a collection of maps $q_\alpha: (X, x_0) \to (W_\alpha, w_\alpha)$ respecting the structure maps $p_{\alpha}^{\alpha'}$. We say that $(\ul{W}, \ul{w}, \ul{q})$ \emph{represents the pointed shape of} $(X, x_0)$ if $\ul{q}$ is initial among maps from $(X, x_0)$ to objects of $\Pro(\Ho\CW_*)$, meaning for any other $(\ul{Z}, \ul{z}) \in \Pro(\Ho\CW_*)$ and morphism $\ul{f}: (X, x_0) \to (\ul{Z}, \ul{z})$ in $\Pro(\Ho\Spc_*)$ there is a unique morphism $\ul{g}: (\ul{W}, \ul{w}) \to (\ul{Z}, \ul{z})$ in $\Pro(\Ho\CW_*)$ making the diagram commute:
    \[\begin{tikzcd}
	(X, x_0) \\
	{(\ul{W}, \ul{w})} & {(\ul{Z}, \ul{z})}
	\arrow["{\ul{q}}"', from=1-1, to=2-1]
	\arrow["{\ul{f}}", from=1-1, to=2-2]
	\arrow["{\ul{g}}"', dashed, from=2-1, to=2-2]
    \end{tikzcd}\]
    This universal property guarantees that triples $(\ul{W}, \ul{w}, \ul{q})$ which represent the pointed shape of $X$ are well-defined up to natural isomorphism in $\Pro(\Ho\CW_*)$, and this isomorphism class is called the \emph{pointed shape of} $X$.
\end{defn}

\begin{rem}\label{rem:equivShapeDef}
    Equivalently, $(\ul{W}, \ul{w}, \ul{q})$ represents the shape of $X$ if and only if any morphism $f: (X, x_0) \to (K, k_0)$ with $(K, k_0)$ a pointed CW complex factors uniquely through $(\ul{W}, \ul{w})$ in $\Pro(\Ho\Spc_*)$. Indeed, since $\Ho\CW_* \seq \Pro(\Ho\CW_*)$ this is implied by the original definition, and conversely given a factorization along single CW complexes and $\ul{f}: (X, x_0) \to (\ul{Z}, \ul{z}) \in \Pro(\Ho\CW_*)$, for each $f_\beta: (X, x_0) \to (Z_\beta, z_\beta)$ we get a $g_\beta: (\ul{W}, \ul{w}) \to (Z_\beta, z_\beta)$ making the diagram commute
    \[\begin{tikzcd}
	X \\
	{(\ul{W}, \ul{w})} & {(Z_\beta, z_\beta)}
	\arrow["{\ul{q}}"', from=1-1, to=2-1]
	\arrow["{f_\beta}", from=1-1, to=2-2]
	\arrow["{g_\beta}"', dashed, from=2-1, to=2-2]
    \end{tikzcd}\]
    which by uniqueness give that the maps $g_\beta \in \Hom((\ul{W}, \ul{w}), (Z_\beta, z_\beta))$ assemble to a map $\ul{g} \in \Hom((\ul{W}, \ul{w}), (\ul{Z}, \ul{z}))$ with $\ul{f} = \ul{g} \circ \ul{q}$ as in the original definition.
\end{rem}

As seen in Remark \ref{rem:equivShapeDef}, the pointed shape of $(X, x_0)$ captures all of the data about mapping out of $(X, x_0)$ and into pointed CW complexes up to homotopy. One way to make this precise is to consider open coverings of $X$ and partitions of unity to construct $\text{\v{C}}(X, x_0)$, the \v{C}ech system of $(X, x_0)$ as in \cite{dydak_segal}. This $\text{\v{C}}(X, x_0)$ is an inverse system in $\Pro(\Ho\CW_*)$, which also has a morphism $\ul{q}_X: X \to \text{\v{C}}(X, x_0)$ in $\Pro(\Ho\Spc_*)$. Theorem 3.1.4 in \cite{dydak_segal} tells us that this pair $(\text{\v{C}}(X,x_0), \ul{q}_X)$ is a representative for the pointed shape of $(X, x_0)$. In the discussion following Corollary 3.1.6 in \cite{dydak_segal} we see that this assignment of \v{C}ech systems is functorial and hence we have functorial representatives for the pointed shape of pointed topological spaces, which we capture in the next lemma.

\begin{lem}\label{lem:shapeExistence}
    There is a functor $\Shape_{\text{\v{C}ech}}: \Spc_* \to \Pro(\Ho\CW_*)$ that to each pointed topological space $(X, x_0)$ assigns its \v{C}ech system $\text{\v{C}}(X,x_0)$, which is a representative for its pointed shape.
\end{lem}
\begin{proof}
    To make the assignment $(X, x_0) \mapsto \text{\v{C}}(X,x_0)$ is functorial we need to define the value of $\Shape_{\text{\v{C}ech}}$ on a morphism $f:(X, x_0) \to (Y, y_0)$ in $\Spc_*$. \\
    As noted above we have that $(\text{\v{C}}(X,x_0), \ul{q}_X)$ and $(\text{\v{C}}(Y,y_0), \ul{q}_Y)$ are representatives for the pointed shape of $(X, x_0)$ and $(Y, y_0)$ respectively. By considering $f:(X, x_0) \to (Y, y_0)$ up to pointed homotopy and then pushing it through the inclusion $\Ho\Spc_* \to \Pro(\Ho\Spc_*)$ we can make the composition $\ul{q}_Y \circ f$ seen in the diagram below.
    \[\begin{tikzcd}
	{(X, x_0)} & {(Y, y_0)} \\
	{\text{\v{C}}(X,x_0)} & {\text{\v{C}}(Y, y_0)}
	\arrow["f", from=1-1, to=1-2]
	\arrow["{\ul{q}_X}"', from=1-1, to=2-1]
	\arrow["{\ul{q}_Y}", from=1-2, to=2-2]
	\arrow["F"', dashed, from=2-1, to=2-2]
    \end{tikzcd}\]
    Applying the universal property of the shape representative $(X, x_0)$ we get a unique morphism $F: \text{\v{C}}(X,x_0) \to \text{\v{C}}(Y, y_0)$ in $\Pro(\Ho\CW_*)$ making the diagram commute. We use this $F$ as the definition for $\Shape_{\text{\v{C}ech}}(f)$. The fact that $\Shape_{\text{\v{C}ech}}$ preserves composition of morphisms follows by using uniqueness in the universal property of the shape representatives.
\end{proof}

If we want to consider the pointed shape of $(X, x_0)$ we should consider $\Shape_{\text{\v{C}ech}}(X, x_0) = \text{\v{C}}(X,x_0)$ up to isomorphism in $\Pro(\Ho\CW_*)$. The notion of isomorphism in $\Pro(\Ho\CW_*)$ can be understood more explicitly through Proposition \ref{prop:inverseProcategory}, though we shall mostly use the universal property of pointed shape representatives to pick alternative representatives $(\ul{W}, w_0, \ul{q})$. Indeed, we are free to change our representative for the shape of a space at will, as we shall see in the following remark.

\begin{rem}\label{rem:shapeCW}
    Consider the case of pointed CW complexes. It follows immediately from Remark \ref{rem:equivShapeDef} that each pointed CW complex $(K, k_0)$ thought of as an inverse system indexed by a single element along with the identity map $(K, k_0, \id_K)$ is a pointed shape representative for $(K, k_0)$. Thus natural functor $\CW_* \to \Ho\CW_* \to \Pro(\Ho\CW_*)$ which first takes morphisms up to homotopy and then includes $\Ho\CW$ into its pro-category, is naturally isomorphic to the restriction of the functor $\Shape_{\text{\v{C}ech}}: \Spc_* \to \Pro(\Ho\CW_*)$ to the subcategory $\CW_*$.
\end{rem}

In this paper we will require a way of computing the shape of a compactum, in order to perform calculations. This is provided to us via the following theorem (Theorem 4.1.5 in \cite{dydak_segal}):

\begin{thm}\label{thm:shapeCompacta}
    Consider an inverse system $(\ul{X}, \ul{x}) = (X_\alpha, x_\alpha, p_{\alpha}^{\alpha'}, A)$ where each $(X_\alpha, x_\alpha)$ is a finite pointed CW complex, and its limit
    \[(X, x_0) = \varprojlim_{\alpha \in A} (X_\alpha, x_\alpha).\]
    Then $(\ul{X}, \ul{x}, \ul{p})$ represents the pointed shape of $(X, x_0)$, where $\ul{p} = \{p_\alpha : X \to X_\alpha\}$ are the structure maps provided by the inverse limit.
\end{thm}

We can combine this result with a theorem of Freudenthal \cite{Freudenthal1937} which says that we may write any metrizable compactum as an inverse limit of a sequence of polyhedra with piecewise-linear maps. \par

Finally we may use this framework to define some topological invariants.

\begin{defn}
    Given any functor which is a pointed homotopy invariant $\cF: \Ho\CW_* \to \Grp$, we can turn it into a shape invariant by defining:
    \begin{align*}
        \ul{\cF}: \Spc_* & \longrightarrow \Pro(\Grp) \\
        (X, x_0) & \longmapsto \Pro(\cF)(\Shape_{\text{\v{C}ech}}(X, x_0))
    \end{align*}
    Note that the shape $\Shape_{\text{\v{C}ech}}(X, x_0)$ lies in $\Pro(\Ho\CW_*)$, and then we may apply the pro-version of $\cF$, which is a functor $\Pro(\cF): \Pro(\Ho\CW_*) \to \Pro(\Grp)$. Analogously, if we have an unpointed homotopy invariant $\cF: \Ho\CW \to \Grp$ we can define the unpointed shape invariant $\ul{\cF}: \Spc \to \Pro(\Grp)$.\par
    In this fashion, we obtain homotopy pro-groups and (co)homology pro-groups $\ppi_n$, $\pH_n$, and $\pH^n$ by applying the above construction to the usual functors for homotopy groups $\pi_n$, homology groups $H_n$, and cohomology groups $H^n$ (note that $\pH^n$ actually takes values in $\Ind(\Grp) = \Pro(\Grp^{op})$). \par
    It is common to distill the invariants to get a group valued invariant. Note that $\pH_n(X)$ (resp. $\pH^n(X)$) consist of a projective (resp. inductive) system of groups, and hence we may take an inverse (resp. direct) limit to get a single group. In this way, we obtain the \v{C}ech homology and cohomology groups:
    \[\check{H}_n(X) = \varprojlim \pH_n(X) \qquad \check{H}^n(X) = \varinjlim \pH^n(X)\]
\end{defn}

\begin{rem}\label{rem:shapeInvCW}
    For a pointed CW complex $(K, k_0)$ we saw in Remark \ref{rem:shapeCW} that there is a natural isomorphism $\Shape_{\text{\v{C}ech}}(K, k_0) \cong (K, k_0)$ (thinking of $(K, k_0)$ as an inverse system indexed by a single element). Then for a pointed homotopy invariant $\cF: \Ho\CW_* \to \Grp$ we have
    \[\ul{\cF}(K, k_0) = \Pro(\cF)(\Shape_{\text{\v{C}ech}}(K, k_0)) \cong \cF(K, k_0)\]
    That is to say, the shape invariant $\ul{\cF}$ restricts to the original $\cF$ when applied to CW complexes (analogously for unpointed invariants). We'll most often use that $\ppi_1(K, k_0) \cong \pi_1(K, k_0)$.
\end{rem}

In our investigation of continua we'll see that the key invariant is the fundamental pro-group $\ppi_1(X, x_0)$, and we will mostly use the other invariants as a way to gain information about $\ppi_1(X, x_0)$ or answer more classical questions (usually posed in terms of $\check{H}^1(X)$). \par 

Later on we will exhibit examples of pointed continua $(X, x_0)$ whose $\ppi_1(X, x_0)$ satisfy properties of interest (for example allowing or prohibiting solutions to certain polynomials over $C(X)$). To do so we need to know given a pro-group $\ul{G}$ whether we may find a pointed continuum $(X, x_0)$ with $\ppi_1(X, x_0) \cong \ul{G}$. The following lemma suffices for our purposes.

\begin{lem}\label{lem:compactumWithGivenProgroup}
    Given a pro-group $\ul{G}$ isomorphic to an object lying in the subcategory $\Pro(\Grp^{\mathrm{fp}})$ of inverse systems of finitely presented groups, there exists a $2$-dimensional pointed continuum $(X, x_0)$ with the fundamental pro-group $\ppi_1(X, x_0) \cong \ul{G}$. Moreover, we may take $(X, x_0)$ to be $1$-dimensional if $\ul{G}$ is isomorphic to an object lying in the subcategory $\Pro(\Grp^{\mathrm{fg}, \mathrm{free}})$ of inverse systems of finitely generated free groups.
\end{lem}
\begin{proof}
Given any finitely presented group $G$ we can form a finite $2$-dimensional pointed CW complex $(X, x_0)$ with $\pi_1(X, x_0) \cong G$ by using a $1$-cell for each generator and a $2$-cell for each relation. If $G$ is a finitely generated free group, then $(X, x_0)$ can be taken to be $1$-dimensional (a wedge of finitely many circles). \par

Since the statement of the theorem is up to isomorphism, we shall assume that $\ul{G}$ is an inverse system lying in $\Pro(\Grp^{\mathrm{fp}})$. By the above, given this inverse system of finitely presented groups $\ul{G} = (G_\alpha, p_{\alpha}^{\alpha'}, A)$ for each $G_\alpha$ we can find a finite $2$-dimensional pointed CW complex $(X_\alpha, x_\alpha)$ whose fundamental group is $G_\alpha$, and for each $p_{\alpha}^{\alpha'}: G_{\alpha'} \to G_\alpha$ we can choose a pointed (cellular) map $\phi_{\alpha}^{\alpha'}: (X_{\alpha'}, x_{\alpha'})  \to (X_\alpha, x_\alpha)$ that induces $p_{\alpha}^{\alpha'}$ on fundamental groups (though this choice is not unique up to homotopy). Then by Theorem \ref{thm:shapeCompacta}, the inverse limit 
\[(X, x_0) = \varprojlim_{\alpha \in A} (X_\alpha, x_\alpha)\]
is a $2$-dimensional pointed continuum with $\ppi_1(X, x) \cong \ul{G}$. If $\ul{G}$ lies in the subcategory $\Pro(\Grp^{\mathrm{fg}, \mathrm{free}})$ then we could have picked each of the $(X_\alpha, x_\alpha)$ to be a wedge of circles, and hence the limit $(X, x_0)$ is also $1$-dimensional.
\end{proof}

\begin{rem}
    While the subcategories $\Grp^{\mathrm{fp}}$ and $\Grp^{\mathrm{fg}, \mathrm{free}}$ are closed under isomorphisms in $\Grp$, this is no longer the case for the subcategories $\Pro(\Grp^{\mathrm{fg}})$ and $\Pro(\Grp^{\mathrm{fg}, \mathrm{free}})$ of $\Pro(\Grp)$ (given simply by applying the $\Pro$ functor). In any case, we will be dealing with cases in which we have explicit inverse systems of finitely presented or finitely generated free groups, but this Lemma can be applied more broadly to inverse systems which are simply isomorphic to objects in $\Pro(\Grp^{\mathrm{fg}})$ and $\Pro(\Grp^{\mathrm{fg}, \mathrm{free}})$ inside of $\Pro(\Grp)$.
\end{rem}

\subsection{Eilenberg-MacLane spaces}

Recall that given $n \geq 1$ and a group $G$, which has to be abelian if $n \geq 2$, we can find pointed spaces $K(G, n)$, called an Eilenberg-MacLane spaces, whose homotopy groups vanish except for $\pi_n(K(G, n), x_0) = G$. Moreover this assignment taking in (discrete) groups and giving pointed spaces is functorial:
\begin{alignat*}{3}
    K(-, 1) & : \Grp && \longrightarrow \Ho\Spc_* \\
    K(-, n) & : \Ab && \longrightarrow \Ho\Spc_* \qquad \text{for} \ n \geq 2
\end{alignat*}
This is nicely described in \cite{McCord_classifying}. We can always take a model for $K(G, n)$ as a pointed CW complex up to homotopy equivalence. We will mostly be interested in $n = 1$.

\vspace{5mm}

One very useful property of Eilenberg-MacLane spaces is that we can completely characterize maps into them up to homotopy. We have the following correspondence: 

\begin{prop}\label{prop:CW maps into Eilenberg-MacLane}\cite[0.5.5]{BauesObstructionTheoryBook}
    Let $(X, x_0)$ be a pointed CW complex. Then the map
    \[\pi_1 : [(X, x_0), (K(G, 1), y_0)] \longrightarrow \Hom_{\Grp}(\pi_1(X, x_0), G)\]
    is a bijection. That is, pointed homotopy classes of maps into $K(G, 1)$ are fully determined by the induced map on fundamental groups.
\end{prop}

If we use a CW complex to model the Eilenberg-MacLane space, then we can upgrade this correspondence to continua.

\begin{lem}\label{lem:maps into Eilenberg-MacLane}
    Let $(X, x_0)$ be a pointed continuum. If $(Y, y_0)$ is a CW model of an Eilenberg-MacLane space $K(G, 1)$, then the map
    \[\ppi_1 : [(X, x_0), (Y, y_0)] \longrightarrow \Hom_{\Pro\Grp}(\ppi_1(X, x_0), G)\]
    is a bijection. That is, pointed homotopy classes of maps into $Y$ are fully determined by the induced map on fundamental pro-groups.
\end{lem}
\begin{proof}
    Take an inverse system $\ul{W} = (W_\alpha, w_\alpha, p_{\alpha}^{\alpha'}, A)$ of pointed connected CW complexes and a pointed morphism $\ul{q}: (X, x_0) \to (\ul{W}, \ul{w})$ so that $(\ul{W}, \ul{q})$ represents the pointed shape of $(X, x_0)$. As in Remark \ref{rem:shapeInvCW} we have that $\ppi_1(Y, y_0) \cong \pi_1(Y, y_0) = G$. \par
    
    We start by proving surjectivity. Given a map $\ul{\phi}: \ppi_1(X, x_0) \to G$, we know it is represented by a group homomorphism $\phi_\beta : \pi_1(W_\beta, w_\beta) \to G$, for some index $\beta \in A$. Now applying Proposition \ref{prop:CW maps into Eilenberg-MacLane}, this homomorphism is induced by a pointed map $f_\beta: (W_\beta, w_\beta) \to (Y, y_0)$. But then $f_\beta \circ q_\beta: (X, x_0) \to (Y, y_0)$ is a map such that $\ppi_1(f_\beta \circ q_\beta) = \ul{\phi}$. \par 
    
    Now for injectivity, consider two pointed maps $f, g: (X, x_0) \to (Y, y_0)$ such that $\ppi_1(f) = \ppi_1(g)$. By applying Remark \ref{rem:equivShapeDef} twice and taking a common upper bound in $A$, there is an index $\beta \in A$ and pointed maps $f_\beta, g_\beta: (W_\beta, w_\beta) \to (Y, y_0)$ such that $f_\beta \circ q_\beta$ is basepoint homotopic to $f$ and $g_\beta \circ q_\beta$ is basepoint homotopic to $g$. Since $\ppi_1(f) = \ppi_1(g)$, there is a $\gamma \geq \beta$ such that the map $\pi_1(f_\beta \circ p_\beta^\gamma) : \pi_1(W_\gamma. w_\gamma) \to G$ is equal to the map $\pi_1(g_\beta \circ p_\beta^\gamma)$. But since the pointed maps $f_\gamma = f_\beta \circ p_\beta^\gamma$ and $g_\gamma = g_\beta \circ p_\beta^\gamma$ from $(W_\gamma, w_\gamma)$ to $(Y, y_0) \cong K(G, 1)$ induce the same map on $\pi_1$, they must be basepoint homotopic by Proposition \ref{prop:CW maps into Eilenberg-MacLane}. But given such a homotopy $f_\gamma \simeq g_\gamma$, we obtain that
    \begin{align*}
        f & \simeq f_\beta \circ q_\beta \simeq f_\beta \circ p_\beta^\gamma \circ q_\gamma  = f_\gamma \circ q_\gamma \simeq g_\gamma \circ q_\gamma = g_\beta \circ p_\beta^\gamma \circ q_\gamma \simeq g_\beta \circ q_\beta \simeq g
    \end{align*}
    and so $f$ and $g$ are basepoint homotopic.
\end{proof}

\section{Polynomials via $\ppi_1(X, x_0)$}\label{sec:polys_via_ppi1}

In Section \ref{subsec:prelim_spaces_of_polys} we saw that a polynomial with no repeated roots on $X$ is nothing but a map $P: X \to B_n$, and in Section \ref{subsec:spaces_of_polys_and_roots} we saw that solving these polynomials amounts to lifting across the covering map $\rho: E_n \to B_n$. Recall also that the spaces $B_n$ and $E_n$ happen to be Eilenberg-MacLane spaces for the groups $\cB_n$ and $M_n$ respectively, so we already know a lot about maps into these spaces. \par

The following Corollary allows us to translate between questions about solving polynomials $P: X \to B_n$ and questions about factoring morphisms in $\Pro(\Grp)$. In particular the group morphisms $M_{n,i} \to \cB_n$ and $N_n \to \cB_n$ given by including subgroups of the braid group play a key role, and in the following we think of these group morphisms as pro-group morphisms via the inclusion functor $\Grp \to \Pro(\Grp)$.

\begin{cor}\label{cor:polys solutions and pro-groups}
    Consider a pointed continuum $(X, x_0)$ and a degree $n$ monic polynomial $b_0 \in B_n$ and a root $\lambda_0 \in E_n$ of $b_0$. 
    \begin{enumerate}[(i)]
        \item For each morphism of pro-groups $\ul{\phi} : \ppi_1(X, x_0) \to \cB_n$, there exists a polynomial $P: (X, x_0) \to (B_n, b_0)$ with $\ppi_1(P) = \ul{\phi}$.
        \item Consider a polynomial $P: (X, x_0) \to (B_n, b_0)$, with its $n$ distinct roots at $x_0$ given as $\{z_1, \dots, z_n\}$. Then $P$ has a solution $\lambda: (X, x_0) \to (E_n, z_i)$ if and only if $\ppi_1(P): \ppi_1(X, x_0) \to \pi_1(B_n, b_0)$ factors through $\pi_1(\rho): \pi_1(E_n, (b_0, z_i)) \cong M_{n,i} \to \cB_n \cong \pi_1(B_n, b_0)$ in $\Pro(\Grp)$. Moreover, $P$ is completely solvable if and only if $\ppi_1(P)$ factors through $\pi_1(\widetilde{\rho}): \pi_1(\widetilde{E}_n, (b_0, z_1, \dots, z_n)) \cong N_n \to \cB_n \cong \pi_1(B_n, b_0)$ in $\Pro(\Grp)$.
    \end{enumerate}
\end{cor}
\begin{proof}
    All are a consequence of Lemma \ref{lem:maps into Eilenberg-MacLane}. Note that $B_n$ and $E_n$ are both smooth manifolds and hence they are CW models for the Eilenberg-MacLane spaces $K(\cB_n, 1)$ and $K(M_n, 1)$ respectively. \par
    To prove (i), we use that the forgetful map
    \[\ppi_1 : [(X, x_0), (B, b_0)] \longrightarrow \Hom_{\Pro\Grp}(\ppi_1(X, x_0), \cB_n)\]
    is surjective to find a polynomial $P$ that maps to $\ul{\phi}$. \par
    For (ii), the ``if" direction is immediate, as $P$ having a solution $\lambda$ means that the diagram
    \[\begin{tikzcd}
        & (E_n, z_i) \\
        (X, x_0) & (B_n, b_0)
        \arrow["P"', from=2-1, to=2-2]
        \arrow["\rho", from=1-2, to=2-2]
        \arrow["\lambda", from=2-1, to=1-2]
    \end{tikzcd}\]
    commutes, which, after applying the functor $\ppi_1$, gives the commutative diagram
    \[\begin{tikzcd}
        & M_{n,i} \\
        \ppi_1(X, x_0) & \cB_n
        \arrow["\ppi_1(P)"', from=2-1, to=2-2]
        \arrow["\pi_1(\rho)", hook, from=1-2, to=2-2]
        \arrow["\ppi_1(\lambda)", from=2-1, to=1-2]
    \end{tikzcd}\]
    showing that $\ppi_1(P)$ factors through $\pi_1(\rho)$. \par
    For the converse, assume that $\ppi_1(P)$ factors through $\pi_1(\rho)$, which means that we have a commutative diagram
    \[\begin{tikzcd}
        & M_{n,i} \\
        \ppi_1(X, x_0) & \cB_n
        \arrow["\ppi_1(P)"', from=2-1, to=2-2]
        \arrow["\pi_1(\rho)", hook, from=1-2, to=2-2]
        \arrow["\ul{\phi}", from=2-1, to=1-2]
    \end{tikzcd}\]
    for some morpshim $\ul{\phi}: \ppi_1(X, x_0) \to M_{n,i}$. Similarly to part (i), we can use surjectivity of the map in Lemma \ref{lem:maps into Eilenberg-MacLane} for the Eilenberg-MacLane space $E_n = K(M_{n,i}, 1)$ to get a pointed map $f: (X, x_0) \to (E_n, z_i)$ with $\ppi_1(f) = \ul{\phi}$. \par
    We see that $\rho \circ f: (X, x_0) \to (B_n, b_0)$ is a pointed map such that
    \[\ppi_1(\rho \circ f) = \pi_1(\rho) \circ \ul{\phi} = \ppi_1(P)\]
    Using injectivity of the map in Lemma \ref{lem:maps into Eilenberg-MacLane} for the Eilenberg-MacLane space $B_n$, we see that $\rho \circ f$ is homotopic to $P$ (relative to the basepoint). Finally, since $\rho$ is a covering map, the property of having a lift across $\rho$ is a homotopy invariant (by lifting the homotopy), so $\rho \circ f$ having a lift $f$ implies that $P$ has a lift $\lambda$ as required. \par 

    Now we consider the statement about complete solvability. The above proof goes through unchanged by switching $\rho: E_n \to B_n$ for $\widetilde{\rho}: \widetilde{E}_n \to B_n$, as lifting across $\widetilde{\rho}$ corresponds to complete solvability. The key properties we used were that $E_n$ is an Eilenberg-MacLane space and $\rho$ is a covering map, which is true of $\widetilde{E}_n$ and $\widetilde{\rho}$ as well.
\end{proof}

This Corollary inspires the following conditions a pro-group $\ul{G}$ can satisfy:
\begin{enumerate}
    \item[$(*_n)$] For each morphism $\ul{\phi}: \ul{G} \to B_n$, there is an $i \in \{1, \dots, n\}$ such that $\ul{\phi}$ factors through the inclusion $M_{n,i} \to \cB_n$ in $\Pro(\Grp)$.
    \item[$(**_n)$] All morphisms $\ul{\phi}: \ul{G} \to B_n$ factor through the inclusion $N_n \to \cB_n$ in $\Pro(\Grp)$. Equivalently, for any morphism $\ul{\phi}: \ul{G} \to B_n$ the composition $\tau \circ \ul{\phi}: \ul{G} \to S_n$ is trivial (as $N_n = \ker(\tau)$ where $\tau : \cB_n \to S_n$ is the canonical map).
\end{enumerate}

We now have complete topological characterizations of when monic polynomials with no repeated roots over $C(X)$ have exact solutions. 

\begin{thm}\label{thm:progroup_char_solvability}
    Consider a continuum $X$, a positive integer $n$, and any basepoint $x_0 \in X$. Then
    \begin{enumerate}[(i)]
        \item Every monic polynomial of degree $n$ over $C(X)$ with no repeated roots has an exact root if and only if $\ppi_1(X, x_0)$ satisfies $(*_n)$.
        \item Every monic polynomial of degree $n$ over $C(X)$ with no repeated roots can be factored completely if and only if $\ppi_1(X, x_0)$ satisfies $(**_n)$.
    \end{enumerate}
\end{thm}

\begin{thm}\label{thm:two_dim_continuum_noexactroot}
    Given a positive integer $n$ and a pro-group $\ul{G}$ which is an inverse system of finitely presented groups, and a morphism $\ul{\phi}: \ul{G} \to B_n$ which does not factor through any of the inclusions $M_{n,i} \to \cB_n$, there exists a $2$-dimensional pointed continuum $(X, x_0)$ with $\ppi_1(X, x_0) \cong \ul{G}$ and a polynomial $P:X \to B_n$ with no repeated roots which has no exact root.
\end{thm}
\begin{proof}
    Given such a pro-group $\ul{G}$, apply Lemma \ref{lem:compactumWithGivenProgroup} to obtain $(X,x_0)$; then apply Theorem \ref{thm:progroup_char_solvability}(i) to obtain the polynomial $P$. 
\end{proof}

\section{Main results and some examples}\label{sec:main_results}

Applying Theorem \ref{thm:progroup_char_solvability} and Corollary \ref{cor:completelySolvableDegN} for each positive integer, we obtain a topological characterization for when the ring of continuous functions $C(X)$ is approximately algebraically closed in terms of the fundamental pro-group, when $\dim X \leq 1$. 

\begin{thm}\label{thm:low_dim_approx_alg_closure_propi1}
    Let $X$ be a continuum with $\dim X \leq 1$, and pick any basepoint $x_0 \in X$. The following are equivalent: 
    \begin{enumerate}[(i)]
        \item $C(X)$ is approximately algebraically closed. 
        \item The fundamental pro-group $\ppi_1(X, x_0)$ satisfies $(*_n)$ for every $n$. 
        \item The fundamental pro-group $\ppi_1(X, x_0)$ satisfies $(**_n)$ for every $n$. 
    \end{enumerate}
\end{thm}

\begin{example}\label{ex:co-ec_continua}
    The notion of a \emph{co-existentially closed continuum}, first introduced by Bankston in \cite{Bankston1999}, is important in model theory. As shown in \cite[Corollary 4.13]{Bankston2006}, co-existentially closed continua are one-dimensional and hereditarily indecomposable. It has recently been shown that co-existentially closed continua are approximately algebraically closed \cite[Theorem 4.6]{EagleLau}. Thus, Theorem \ref{thm:low_dim_approx_alg_closure_propi1} applies, and we obtain a new topological property of co-existentially closed continua, concerning their fundamental pro-group. 
\end{example}

We now show that in order to solve certain special polynomials, namely those of the form $z^m - f \in C(X)[z]$ for $f$ non-vanishing, it is enough to look at the \v{C}ech cohomology of $X$. There are some classical results are stated in terms of the divisibility of $\check{H}^1(X)$, for which we now provide new, updated proofs using shape-theoretic invariants. We start with a small proposition.

\begin{prop}\label{prop:Divisibility of dual of progroup and subgroups}
    Given a pro-group $\ul{G}$, consider the dual group
    \[A = \varinjlim \Hom(\ul{G}, \bZ)\]
    Then $A$ is $m$-divisible if and only if any morphism $\ul{\phi}: \ul{G} \to \bZ$ factors through the inclusion of the subgroup $m\bZ \to \bZ$ in $\Pro(\Grp)$ (and hence morphisms into $\bZ^k$ factor through the inclusion $m^a\bZ^k \to \bZ^k$ for each $a, k \in \bN$, by working one coordinate at a time and by induction on $a$).
\end{prop}
\begin{proof}
A morphism $\ul{\phi}: \ul{G} \to \bZ$ is exactly an element of $A$, and $\ul{\phi}$ factors through $m\bZ \to \bZ$ if and only if there exists some index $\alpha$ such that the image of a representative $\phi_\alpha : G_\alpha \to \bZ$ lands in $m\bZ$, if and only if there exists some index $\alpha$ and homomorphism $\psi_\alpha : G_\alpha \to \bZ$ such that $\phi_\alpha = m\psi_\alpha$, if and only if there exists some $\ul{\psi}: \ul{G} \to \bZ$ with $m \ul{\psi} = \ul{\phi}$, which is the definition of $A$ being $m$-divisible.
\end{proof}

With this we can analyze $m$th roots of non-vanishing functions, and we arrive at a well-known statement in the spirit of \cite[Theorem 1.3]{KawamuraMiura}.

\begin{cor}\label{cor:divisibility and mth roots}
    Given a continuum $X$, we get that $\check{H}^1(X)$ is $m$-divisible if and only if for any non-vanishing $f \in C(X)$ there exists a $g \in C(X)$ such that $g^m = f$.
\end{cor}
\begin{proof}
    Consider the subset $C_m \seq B_m$ consisting of polynomials $z^m - \mu$ for $\mu \in \bC^\times$, which by this parametrization is homeomorphic to $\bC^\times$. Then the preimage of $C_m$ under our usual covering map $\rho: E_m \to B_m$ is again a copy of $\bC^\times$, and the covering map is given by $\lambda \mapsto \lambda^m$.
    Then given such an $f$ we get a polynomial $P_f: X \to C_m \seq B_m$ by $P_f(x, z) = z^m - f(x)$ and the question of asking for a solution $g$ is equivalent to asking for a lift on fundamental groups:
    \[\begin{tikzcd}
	& {\bC^\times \seq E_m} &&& {m\bZ} \\
	X & {C_m \seq B_m} && {\ppi_1(X,x_0)} & \bZ
	\arrow[from=1-2, to=2-2]
	\arrow[hook, from=1-5, to=2-5]
	\arrow[dashed, from=2-1, to=1-2]
	\arrow["{P_f}"', from=2-1, to=2-2]
	\arrow[dashed, from=2-4, to=1-5]
	\arrow["{\ppi_1(P_f)}"', from=2-4, to=2-5]
    \end{tikzcd}\]
    Now, by Proposition \ref{prop:Divisibility of dual of progroup and subgroups} since the dual group
    \[\check{H}^1(X) = \varinjlim \Hom(\ppi_1(X,x_0), \bZ)\]
    is $m$-divisible, we get that $\ppi_1(P_f)$ factors through $m\bZ \to \bZ$ as required. \par
    
    Conversely, given a morphism $\ul{\phi}: \ppi_1(X,x_0) \to \bZ$ we know it is represented by some $\phi_\alpha: \pi_1(W_\alpha, w_\alpha) \to \bZ$, where $\ul{q}: (X, x_0) \to (\ul{W}, \ul{w})$ represents the pointed shape of $(X,x_0)$. But $C_m \cong \bC^\times$ is an Eilenberg-MacLane space $K(\bZ, 1)$, so this $\phi_\alpha$ is represented by a pointed map $h_\alpha: (W_\alpha, w_\alpha) \to (\bC^\times, 1)$ and $P_f = h_\alpha \circ q_\alpha: X \to C_m$ is a polynomial $P_f(z, x) = z^m - f(x)$ with $\ppi_1(P_f) = \ul{\phi}$. By assumption $f$ has an $m^{\text{th}}$ root $g$, which by the above reasoning means that $\ppi_1(P_f) = \ul{\phi}$ factors through $m\bZ \to \bZ$. So $\check{H}^1(X)$ is $m$-divisible by Proposition \ref{prop:Divisibility of dual of progroup and subgroups}.
\end{proof}

The next theorems are about using $m^{\text{th}}$ roots to solve more complicated polynomials. All of these results are about assuming some amount of solvability of $\ppi_1(X, x_0)$. The first is a restatement of \cite[Theorem 1.8]{GorinLin} in our language of fundamental pro-groups, dealing with the case that the fundamental pro-group is abelian.

\begin{thm}\label{thm:low_dim_approx_alg_closure_cechH1}
    Let $X$ be a continuum with $\dim X \leq 1$ such that for some $x_0 \in X$ the fundamental pro-group $\ppi_1(X, x_0)$ is abelian. The following are equivalent:
    \begin{enumerate}[(i)]
        \item $C(X)$ is approximately algebraically closed. 
        \item The first \v{C}ech cohomology group $\check{H}^1(X)$ is divisible. 
    \end{enumerate}
\end{thm}
\begin{proof}
    We know that $\ppi_1(X, x_0)$ is isomorphic to an inverse sequence of abelian groups by Proposition \ref{prop:abelian_pro_groups-abelian_inverse_systems}. By Corollary \ref{cor:completelySolvableDegN}, (i) is equivalent to the assertion that every monic polynomial with coefficients in $C(X)$ with no repeated roots can be factored completely. But by \cite[Theorem 1.8]{GorinLin}, since $\ppi_1(X, x_0)$ is isomorphic to an inverse sequence of abelian groups, this is equivalent to (ii). 
\end{proof}

\begin{examples}\label{ex:variousExamplesofApproxAlgClosure_cechH1}
Using Theorem \ref{thm:low_dim_approx_alg_closure_cechH1} above, we can describe some examples and non-examples of approximately algebraically closed continua. 
    \begin{enumerate}
        \item A tree-like continuum, i.e.~an inverse limits of graphs, is approximately algebraically closed, because its $\check{H}^1$ vanishes. 
        \item A solenoid $\Sigma$ is approximately algebraically closed if and only if it is the universal one (i.e. has $\check{H}^1(\Sigma; \mathbb{Z}) = \mathbb{Q}$). 
        \item A pseudo-solenoid $\mathbb{P}\Sigma$ is shape-equivalent to $\Sigma$, so by (2), $\mathbb{P}\Sigma$ is approximately algebraically closed if and only if it is the universal one. 
    \end{enumerate}
Descriptions and definitions of solenoids and pseudo-solenoids may be found in \cite[Section 5]{EagleLau}. It was observed in \cite[Corollary 3.4]{KawamuraMiura2009} that a solenoid or pseudo-solenoid which is not universal cannot be approximately algebraically closed; Theorem \ref{thm:low_dim_approx_alg_closure_cechH1} establishes the converse of this observation, and gives further indication that the universal pseudo-solenoid may be co-existentially closed, which is partial progress towards answering \cite[Problem 5.15]{EagleLau}.
\end{examples}

\begin{example}
    There exists a one-dimensional pointed continuum $(X, x_0)$ which is acyclic (i.e. $\check{H}^1(X) = 0$), but for which the fundamental pro-group $\ppi_1(X, x_0)$ is non-abelian (i.e. is not isomorphic to an inverse limit of abelian groups) and satisfies $(*_n)$ for every $n \geq 1$. In particular $C(X)$ is approximately algebraically closed, but we cannot use the criterion in Theorem \ref{thm:low_dim_approx_alg_closure_cechH1} to determine this, and must use the full power of the main theorem. \par
    To produce such an $X$, we will first construct a sequence of nested subgroups
    \begin{align*}
        G_1 \supseteq G_2 \supseteq G_3 \supseteq \cdots
    \end{align*}
    with the following properties:
    \begin{enumerate}
        \item $G_n$ is a free group on two generators for every $n$;
        \item $G_{n+1}$ is contained in the commutator subgroup $G_n'$ of $G_{n}$ for each $n$; and, 
        \item Given any group homomorphism $\varphi : G_n \to \cB_k$ from any group $G_n$ in the sequence to any braid group $\cB_k$, there is some integer $m \geq n$ such that $G_m$ lies in the kernel of $\tau \circ \varphi$, where $\tau : \cB_k \to S_k$ is the canonical map. 
    \end{enumerate}
    We will construct these groups recursively, and at each stage of the recursion, we will not only produce a group $G_n$, but we will also keep track of an enumeration of the countable set $\bigcup_{k=1}^{\infty}\Hom(G_n,\cB_k)$. This enumeration will help build the groups further down in the sequence. We will also let $(\ell_i)_{i =1}^{\infty}$ be a sequence of positive integers with the property that each positive integer appears infinitely often in the sequence and $\ell_i \leq i$ for all $i$. For example, we could take $(\ell_i)$ to be the sequence $1,1,2,1,2,3,1,2,3,4, \ldots$. \par
    For the initial step in the recursion, set $G_1 = \mathbb{Z} * \mathbb{Z}$, and choose an enumeration $(\varphi_{1,j})_{j=1}^{\infty}$ of the set $\bigcup_{k=1}^{\infty} \Hom(G_1, \cB_k)$. Note that this set is actually countable because $\cB_k$ is countable and $G_1$ is finitely generated. For the recursive step, suppose that for some $n \geq 1$ we have constructed the subgroups $G_m$ and the enumerations $(\varphi_{m,j})_{j=1}^{\infty}$ of $\bigcup_{k=1}^{\infty} \Hom(G_m, \cB_k)$ for all $1 \leq m \leq n$. Let $j_n:= \#\{ i \in \mathbb{N} : 1 \leq i \leq n \text{ and } \ell_i = \ell_n\}$. Consider the map $\psi_n : G_n' \to S_k$ given by the composition of the following maps:
    \[\begin{tikzcd}
	{G_n'} & {G_n} & {G_{\ell_n}} & {\mathcal{B}_k} & {S_k}
	\arrow[hook, from=1-1, to=1-2]
	\arrow[hook, from=1-2, to=1-3]
	\arrow["{\varphi_{\ell_n,j_n}}", from=1-3, to=1-4]
	\arrow["\tau", from=1-4, to=1-5]
    \end{tikzcd}\]
    Here, the first two maps are inclusions, and $\tau$ is the canonical map (note that $k$ depends on $\ell_n$ and $j_n$). Since $G_n$ is a free group on two generators, its commutator subgroup $G_n'$ is also free of countably infinite rank. Thus if $H_{n+1} := \ker \psi_n$, then $H_{n+1}$ must be a countably infinite free group, since $[G_n' : H_{n
    +1}] = |S_k| = k!$ is finite. Moreover, the rank of $H_{n+1}$ must be at least two; otherwise, $H_{n+1}$ would be contained as a finite index subgroup of a rank two subgroup of $G_{n}'$, which by the Nielsen-Schreier Theorem would imply that $[G_n' : H_{n+1}] = 0$, a contradiction. Thus, we can choose a rank two subgroup of $H_{n+1}$, and define $G_{n+1}$ to be this subgroup. Then $G_{n+1}$ is finitely generated, so the set $\bigcup_{k=1}^{\infty} \Hom(G_{n+1}, \cB_k)$ is countable, and we can choose an enumeration $(\varphi_{n+1,j})_{j=1}^{\infty}$ for this set. This completes the recursion. \par
    Now notice that we have constructed a sequence $G_1 \supseteq G_2 \supseteq G_3 \supseteq \cdots$ of nested subgroups, each of which is a free group on two generators (so property (1) is satisfied). Property (2) is also satisfied by construction, since $G_{n+1}$ is contained in the kernel of a map with domain $G_{n}'$. Finally, given a group homomorphism $\varphi : G_n \to \cB_k$ for some $n$ and some $k$, there is a positive integer $j$ such that $\varphi = \varphi_{n,j}$. Since every positive integer appears in the sequence $(\ell_i)$ infinitely often, there is some smallest positive integer $i_0$ such that $j = \#\{ i \in \mathbb{N} : 1 \leq i \leq i_0 : \ell_i = n\}$; note that $n = \ell_{i_0} \leq i_0$. Then by construction, at the $i_0^{\text{th}}$ step of the recursion, we produce a group $G_{i_0 + 1}$ which is contained in the kernel of the map $\tau \circ \varphi_{\ell_{i_0}, j_{i_0}} = \tau \circ \varphi_{n, j} = \tau \circ \varphi$, as desired. This shows property (3) is satisfied. \par
    This nested sequence of subgroups defines a pro-group $\ul{G}$ given by the inverse system of inclusions
    \[\begin{tikzcd}
	G_1 & G_2 & G_3 & \cdots
	\arrow[from=1-2, to=1-1]
	\arrow[from=1-3, to=1-2]
	\arrow[from=1-4, to=1-3]
    \end{tikzcd}\]
    such that, by property (2), each structure map factors through the commutator subgroup of its range. In particular, this implies that the abelianization of $\ul{G}$ is the trivial group. Moreover, property (3) guarantees that the pro-group $\ul{G}$ satisfies $(**_k)$, and hence $(*_k)$, for every $k$. Finally, applying Lemma \ref{lem:compactumWithGivenProgroup}, we can find a 1-dimensional pointed continuum $(X, x_0)$ having $\ppi_1(X,x_0) \cong \ul{G}$, which can be realized as an inverse limit of wedges of two circles. Then $\ppi_1(X, x_0)$ satisfies $(*_k)$ for every $k$, and it is non-abelian, since the abelianization of $\ul{G}$ is trivial (but $\ul{G}$ is non-trivial by Proposition \ref{prop:trivial_progroup_nested_subgroups}, since the structure maps are all inclusions of non-trivial subgroups). Moreover, because $\Ab(\ul{G}) = 0$, we have that $\ul{H}_1(X) \cong \Ab(\ppi_1(X, x_0)) \cong \Ab(\ul{G}) = 0$, so by the UCT,
    \[\check{H}^1(X) = \varinjlim \ul{H}^1(X) \cong \varinjlim \Hom(\ul{H}_1(X), \bZ) = 0\]
    and hence $X$ is acyclic. 
\end{example}

\section{Low-degree polynomials and braid groups}\label{sec:lowdegpolys}
Our next results concern using more easily computable invariants, such as the first \v{C}ech cohomology group $\check{H}^1(X)$ discussed above, and the homology pro-group $\ul{H}_1(X)$, in order to discern if low degree polynomials have continuous approximate solutions. For these results, we use Definition \ref{defn:divisibility_progroup} regarding $m$-divisibility for a pro-group.

\begin{lem}\label{lem:divisibility and solvability}
    Consider a pro-group $\ul{G}$ and its abelianization $\Pro(\Ab)(\ul{G})$. If $W$ is a solvable group of exponent $m$ and $\Pro(\Ab)(\ul{G})$ is $m$-divisible, then any morphism $\ul{\phi}: \ul{G} \to W$ is trivial.
\end{lem}
\begin{proof}
    Since $W$ is a finite solvable group, it has a subnormal series
    \[1 = W_0 \trianglelefteq W_1 \trianglelefteq \cdots \trianglelefteq W_{n - 1} \trianglelefteq W_n = W\]
    such that $W_k/W_{k - 1}$ is abelian for each $1 \leq k$. We shall prove the statement by induction on the length $n$ of this series. If $n = 0$ the statement is trivial as $W$ is already the trivial group. \par
    Now assuming that the statement is true for groups with such a series of length $n - 1$, take $W$ to have such a series of length $n$. Then $W/W_{n - 1} = W_n/W_{n - 1}$ is an abelian group of exponent (dividing) $m$. Hence if
    \[\phi_\alpha: G_\alpha \to W\]
    represents $\ul{\phi}$, then the composition with the quotient $W \to W/W_{n - 1}$ factors through $\Ab(G_\alpha)$, and using the definition of $\Pro(\Ab)(\ul{G})$ being $m$-divisible we can find $\beta \geq \alpha$ so that the image of $\Ab(G_\beta)$ in $\Ab(G_\alpha)$ consists only of $m^{\text{th}}$ powers of elements in $\Ab(G_\alpha)$. Thus the composition
    \[\begin{tikzcd}
	{G_\beta} & {\Ab(G_\beta)} \\
	{G_\alpha} & {\Ab(G_\alpha)} & {W/W_{n - 1}}
	\arrow[from=1-1, to=1-2]
	\arrow[from=1-1, to=2-1]
	\arrow[from=1-2, to=2-2]
	\arrow["0", from=1-2, to=2-3]
	\arrow[from=2-1, to=2-2]
	\arrow[from=2-2, to=2-3]
    \end{tikzcd}\]
    is trivial. Hence the map from $G_\beta$ to $W/W_{n - 1}$ is trivial, and so $\phi_\beta: G_\beta \to W$ must actually have its image land in $W_{n - 1}$. But $W_{n - 1}$ has a shorter series and therefore the map $\ul{\phi}$ is trivial.
\end{proof}

\begin{thm}\label{thm:divisibility of homology pro-group and solvability}
    Consider a continuum $X$ and an integer $1 \leq n \leq 4$. If $\pH_1(X)$ is $n!$-divisible, then all polynomials $P: X \to B_n$ are completely solvable.
\end{thm}
\begin{proof}
    Picking a basepoint $x_0 \in X$, by Theorem \ref{thm:progroup_char_solvability} we need to show $\ppi_1(X, x_0)$ satisfies $(**_n)$, meaning that any $\ul{\phi}: \ppi_1(X, x_0) \to \cB_n$ is such that $\ul{\psi} = \tau \circ \ul{\phi}$ is trivial, where $\tau : \cB_n \to S_n$ denotes the canonical map. But since $n \leq 4$ we have that $\ul{\psi}: \ppi_1(X, x_0) \to S_n$ is a morphism into a solvable group of order $n!$, so By Lemma \ref{lem:divisibility and solvability} the abelianization $\pH_1(X)$ of $\ppi_1(X, x_0)$ being $n!$-divisible guarantees any morphism into $S_n$ is trivial.
\end{proof}

Next we show that for polynomials of degree less than $3$, we can replace divisibility of $\pH_1(X)$ with divisibility of $\check{H}^1(X)$, but not for degree $n = 4$ polynomials. For these proofs we need to know more about the structure of the braid group $\cB_n$. Recall that $\cB_1$ is defined to be the trivial group, while in general, the braid group on $n$ strands (for $n\geq 2$) has a presentation given by generators $\sigma_1, \dots, \sigma_{n - 1}$ and relations
\begin{align*}
    \sigma_i \sigma_j & = \sigma_j \sigma_i \quad \text{for} \ \text 1 \leq i \leq j - 2 \leq n - 3 \\
    \sigma_i \sigma_{i + 1} \sigma_i & = \sigma_{i + 1} \sigma_i \sigma_{i + 1} \quad \text{for} \ \text 1 \leq i \leq n - 2
\end{align*}
Thus $\cB_2$ has a single generator and no relations, so $\cB_2 \cong \bZ$. It is not too hard to see that these relationships imply that the abelianization $\Ab(\cB_n)$ is $\bZ$ with the abelianization homomorphism given by $\sigma_i \mapsto 1$ for each $1 \leq i \leq n - 1$. For our proofs we need to know a bit more about the derived series of $\cB_n$, which is given to us by \cite[Theorems 2.1, 2.6; Corollary 2.2]{GorinLin}:
\begin{itemize}
    \item \underline{For $\cB_3$:} the commutator subgroup $\cB_3'$ is free, generated by the two elements 
    \[u = \sigma_2 \sigma_1^{-1}, \quad v = \sigma_1 \sigma_2 \sigma_1^{-2}\]

    \item \underline{For $\cB_4$:} the commutator subgroup $\cB_4'$ has its presentation given by four generators
        \[u, \quad v, \quad a = \sigma_3 \sigma_1^{-1}, \quad b = uau^{-1}\]
    and relations
        \[uau^{-1} = b, \quad ubu^{-1} = b^2a^{-1}b, \quad vav^{-1} = a^{-1}b, \quad vbv^{-1} = (a^{-1}b)^3 a^{-2}b\]
    where $u$ and $v$ are as above. The additional relations
        \[u^{-1}au = ab^{-1}a^2, \quad u^{-1}bu = a, \quad v^{-1}av = ab^{-1}a^3, \quad v^{-1}bv = ab^{-1}a^4\]
    hold, making the subgroup $T \seq \cB_4'$ generated by $a$ and $b$ a normal subgroup. In fact $T$ is freely generated by $a$ and $b$, and the quotient $\cB_4'/T$ is a free group generated by the images of $u$ and $v$.

    \item \underline{For $\cB_n$ with $n \geq 5$:} the commutator subgroup $\cB_5'$ is perfect, meaning the second commutator subgroup $\cB_5''$ is equal to the first $\cB_5'$.
    
\end{itemize}

Recall from Section \ref{subsec:prelim_spaces_of_polys} that we have an action $\star: \bC^\times \times B_n \to B_n$ given by scaling roots of a polynomial at a point by a non-zero complex number. Then if we consider a continuum $X$, a degree $n$ polynomial $P: X \to B_n$, and a non-vanishing function $f: X \to \bC^\times$, we can define $f \star P$ by applying this action pointwise. Observe that finding solutions to $P$ and $f\star P$ are equivalent problems. In particular, if we can find a non-vanishing function $f$ such that the discriminant of $f\star P$ is constant and equal to $1$, then we can find solutions to $P$ by finding solutions to $f\star P$. This is useful, because if $f\star P :X \to B_n$ has constant discriminant equal to $1$, then its image lies in the subspace $B_n'$ of $B_n$, which has fundamental group $\cB_n'$ (see Section \ref{sec:Preliminaries}), so the induced map $\ppi_1 (f\star P) : \ppi_1(X,x_0) \to \cB_n$ on fundamental pro-groups factors through the inclusion $\cB_n' \to \cB_n$. This is helpful in the context of Corollary \ref{cor:polys solutions and pro-groups} for finding solutions to $f\star P$ (and hence solutions of $P$). \par 
To find such an $f$, recall that $\Delta(f \star P) = f^{n(n - 1)} \Delta(P)$, so we need to make $f$ an $(n(n-1))^{\text{th}}$ root of the function $\frac{1}{\Delta(P)} : X \to \mathbb{C}^\times$. Such an $f$ exists if $\check{H}^1(X)$ is $n(n-1)$-divisible, by Corollary \ref{cor:divisibility and mth roots}. \par

Using the above discussion and the shape-theoretic formalism, we can now give an alternative proof of \cite[Theorem 3.3]{GorinLin} which asserts that for quadratic and cubic polynomials, $2$- and $3$-divisibility of $\check{H}^1(X)$ is enough to find solutions.

\begin{thm}\label{thm:divisibility of cech cohomology implies solvability for degree 2 and 3}
    For $n = 2$ or $n = 3$, a continuum $X$ has an $n!$-divisible $\check{H}^1(X)$ if and only if any polynomial $P: X \to B_n$ is completely solvable.
\end{thm}
\begin{proof}
    For the backwards direction, we can get $2$- or $3$-divisibility of $\check{H}^1(X)$ by using the fact that we can find roots of non-vanishing functions (Corollary \ref{cor:divisibility and mth roots}), so let's look at the forwards direction. \par
    
    The case $n = 2$ is simple. As we argued above, $\check{H}^1(X)$ being $2$-divisible ensures that we only need to worry about morphisms $\ul{\phi}: \ppi_1(X, x_0) \to \cB_2$ which factor through $0 = \cB_2' \to \cB_2$ which means $\ul{\phi}$ is the zero morphism. These definitely factor through $N_2$, so $\ppi_1(X, x_0)$ satisfies $(**_2)$. \par 
    The case $n = 3$ is trickier. Consider some $\ul{\phi}: \ppi_1(X, x_0) \to \cB_3$, which we may similarly assume factors through $\cB_3' \to \cB_3$, as $\check{H}^1(X)$ is $6$-divisible (we'll also use $\ul{\phi}$ for the factoring map $\ppi_1(X, x_0) \to \cB_3'$). Since the image of the commutator subgroup $\cB_3'$ under the canonical map $\tau: \cB_3 \to S_3$ lies in the abelian subgroup $S_3' = A_3 \cong \mathbb{Z}/3\mathbb{Z}$, it follows that $\tau|_{\cB_3'}$ factors through the abelianization of $\cB_3'$. But by the discussion above, $\cB_3'$ is freely generated by two elements, so the abelianization of $\cB_3'$ is $\mathbb{Z}^2$. In total we get the diagram:
    \[\begin{tikzcd}
	& {3\bZ^2} \\
	{\ppi_1(X, x_0)} & {\bZ^2} & {A_3 \cong \bZ/3\bZ} \\
	& {\cB_3'}
	\arrow[hook, from=1-2, to=2-2]
	\arrow["0", from=1-2, to=2-3]
	\arrow[dashed, from=2-1, to=1-2]
	\arrow["{\ul{\phi}}", from=2-1, to=3-2]
	\arrow[from=2-2, to=2-3]
	\arrow[from=3-2, to=2-2]
	\arrow["{\tau|_{\cB_3'}}"', from=3-2, to=2-3]
    \end{tikzcd}\]
    By our assumption that the dual group
    \[\check{H}^1(X) = \Hom(\ppi_1(X, x_0), \bZ)\]
    is $3$-divisible, by Proposition \ref{prop:Divisibility of dual of progroup and subgroups} we see that the the map $\ppi_1(X, x_0) \to \bZ^2$ given by $\ul{\phi}$ and then abelianization, factors through the inclusion $3\bZ^2 \to \bZ^2$. Therefore $\tau \circ \ul{\phi}$ is the trivial map, which in total shows that $\ppi_1(X, x_0)$ satisfies $(**_3)$ as required.
\end{proof}

Here are some counterexamples in the cases $n = 4$ and $n \geq 5$. First the simpler case, $n \geq 5$ when $A_n$ is not solvable (and moreover perfect).

\begin{lem}\label{lem:divisibility of cech cohomology counterexample for degree >= 5}
    Given any integer $n \geq 5$ there exists a pro-group $\ul{G^n}$ given by an inverse system of free groups on two generators
    \[\begin{tikzcd}
	{\bZ * \bZ} & {\bZ * \bZ} & {\bZ * \bZ} & \cdots
	\arrow["f_n"', from=1-2, to=1-1]
	\arrow["f_n"', from=1-3, to=1-2]
	\arrow[from=1-4, to=1-3]
    \end{tikzcd}\]
    that does not satisfy $(*_n)$, but has $\Pro(Ab)(\ul{G^n}) = 0$.
\end{lem}
\begin{proof}
    For $n \geq 5$ we know that the group $A_n$ is perfect and is generated by two elements $a, b \in A_n$. In particular we have a surjective morphism $\psi: \bZ * \bZ \to A_n$ by sending the free generators $x, y$ of $\bZ * \bZ$ to $a$ and $b$ respectively. But now we get a commutative diagram
    \[\begin{tikzcd}
	{\bZ * \bZ} & {(\bZ * \bZ)'} \\
	{A_n} & {A_n'}
	\arrow["\psi"', from=1-1, to=2-1]
	\arrow[from=1-2, to=1-1]
	\arrow[from=1-2, to=2-2]
	\arrow["\id"', from=2-2, to=2-1]
    \end{tikzcd}\]
    obtained by restricting to the commutator subgroups, where the bottom map is the identity (as $A_n'$ is all of $A_n$). Therefore the restriction of $\psi$ to $(\bZ * \bZ)'$ is still surjective, so we can find $x', y' \in (\bZ * \bZ)'$ such that $\psi(x') = a$ and $\psi(y') = b$. We take $f_n: \bZ * \bZ \to \bZ * \bZ$ to be the endomorphism sending $x$ to $x'$ and $y$ to $y'$. Then by construction $\psi \circ f_n = \psi$, and hence $\psi \circ f_n^{\circ m} = \psi$ for any natural number $m$.\par 

    To define $\phi$ we lift $\psi: \bZ * \bZ \to A_n$ along the natural map $\tau|_{\cB_n'}: \cB_n' \to A_n$, which is possible as $\tau$ (and hence $\tau|_{\cB_n'}$) is surjective and $\bZ * \bZ$ is free. Now we take $\ul{G^n}$ to be the inverse system where all of the structure maps are $f_n$, and take the morphism $\ul{\phi}: \ul{G^n} \to \cB_n' \seq \cB_n$ given by $\phi$ from the first object of the inverse system. \par
    If $\ul{\phi}$ factored through some $M_{n,i} \to \cB_n$, then we should be able to factor $\phi$ through $M_{n,i}$ by going far enough up the inverse system, meaning $\phi \circ f_n^{\circ m}$ would have to have its image in $M_{n,i}$ for some large enough $m$. However $M_{n,i}$ is exactly the preimage under $\tau$ of the subgroup $S_{n - 1}$ of $S_n$ consisting of permutations that fix the $i$th element, so it is enough to check that $\tau \circ \phi \circ f_n^{\circ m}$ never lies in these subgroups. But
    \[\tau \circ \phi \circ f_n^{\circ m} = \psi \circ f_n^{\circ m} = \psi\]
    which has an image of $A_n$, and no element is fixed by all permutations in this subgroup. Therefore $\ul{G^n}$ does not satisfy $(*_n)$. \par
    Lastly we need to argue that $\Pro(\Ab)(\ul{G^n}) = 0$, but this is immediate as $f_n: \bZ * \bZ \to \bZ * \bZ$ has its image landing in the commutator subgroup, so $f_n$ becomes the zero map after abelianizing and hence the inverse system consists of a sequence of zero maps.
\end{proof}

\begin{example}\label{ex:zero cech cohomology with poly of deg >= 5 that has no solutions}
As in Lemma \ref{lem:compactumWithGivenProgroup} we can find a $1$-dimensional pointed continuum $(X, x_0)$ which has $\ppi_1(X, x_0) \cong \ul{G^n}$ from Lemma \ref{lem:divisibility of cech cohomology counterexample for degree >= 5} by taking the inverse limit of wedges of two circles
\[\begin{tikzcd}
	{S^1 \vee S^1} & {S^1 \vee S^1} & {S^1 \vee S^1} & \cdots
	\arrow["{F_n}"', from=1-2, to=1-1]
	\arrow["{F_n}"', from=1-3, to=1-2]
	\arrow[from=1-4, to=1-3]
\end{tikzcd}\]
using a pointed map $F_n$ that realizes the morphism $f_n$. This continuum $X$ carries a polynomial $P: X \to B_n$ with no solutions by Theorem \ref{thm:two_dim_continuum_noexactroot}, because $\ppi_1(X, x_0)$ does not satisfy $(*_n)$. On the other hand, this $X$ has $\ul{H}_1(X) \cong \Pro(\Ab)(\ppi_1(X, x_0)) \cong \Pro(\Ab)(\ul{G^n}) = 0$ so by the UCT,
\[\check{H}^1(X) = \varinjlim \ul{H}^1(X) \cong \varinjlim \Hom(\ul{H}_1(X), \bZ) = 0\]
showing in particular that $\check{H}^1(X)$ is divisible. Note that by \cite[Theorem 1.3]{KawamuraMiura}, this continuum $X$ is an example of a continuum that admits approximate $m^{\text{th}}$ roots for every $m \geq 1$, but does not admit approximate continuous solutions to some degree $n$ polynomial. 

\end{example}

\begin{rem}
    Note that an acyclic one-dimensional continuum will always satisfy $(*_4)$. Indeed, if $X$ is one-dimensional, then its shape is given by an inverse sequence of wedges of circles, and hence the homology pro-group $\ul{H}_1(X)$ is an inverse limit of direct sums of copies of $\mathbb{Z}$. If additionally $X$ is acyclic, then  by the UCT, $\varinjlim \Hom(\ul{H}_1(X), \mathbb{Z}) = \check{H}^1(X) = 0$, which implies that $\ul{H}_1(X)$ is $m$-divisible for all $m \geq 1$. In particular, it is $4!$-divisible, and hence Theorem \ref{thm:divisibility of homology pro-group and solvability} applies. 
    
\end{rem}

We now move to the more difficult case of find an acyclic two-dimensional continuum whose fundamental pro-group does not satisfy $(*_4)$. This shows that \cite[Theorem 3.1]{GorinLin} is not true when we remove the hypothesis that the continuum under consideration is a finite cell complex. This improves upon \cite[Theorem 3.4]{GorinLin}, which shows that there exists a two-dimensional continuum whose fundamental pro-group does not satisfy $(*_4)$ but whose first \v{C}ech cohomology group is divisible. 

We start with a technical lemma that will help us build the appropriate pro-group.

\begin{lem}\label{lem:action of free subgroups of sl2Z}
    Consider two elements $U, V \in \SL_2(\bZ)$. The subgroup they generate in $\SL_2(\bZ)$ is freely generated by them if and only if their images in $\PSL_2(\bZ)$ also freely generate a subgroup. In this case, the sum of the ranges
    \[\Im(U - \id) + \Im(V - \id) \seq \bZ^2\]
    is a rank $2$ subgroup.
\end{lem}
\begin{proof}
    Let us denote the quotient map by $f:\SL_2(\bZ) \to \PSL_2(\bZ)$, which has a kernel of $\{\pm \id\}$. Let's start with the equivalence between generating a free subgroup of $\SL_2(\bZ)$ and $\PSL_2(\bZ)$. \par
    If $U, V$ are free generators of their subgroup $\sgen{U, V}$, then we cannot have $-\id$ as an element of $\sgen{U, V}$ as this subgroup is torsion-free. Therefore $\sgen{U, V} \cap \ker(f) = \{\id\}$ meaning that $f$ restricts to be injective on $\sgen{U, V}$ as required. Conversely, if $f(U), f(V)$ freely generate a subgroup in $\PSL_2(\bZ)$ it cannot be that there are any relations between $U$ and $V$, as these would give relations between $f(U)$ and $f(V)$. \par
    
    Next we show that under these conditions, the sum of the ranges $\Im(U - \id) + \Im(V - \id) \seq \bZ^2$ is a rank $2$ subgroup. First note that it cannot be that $U = \id$ or $V = \id$ if they freely generate a subgroup, therefore $\Im(U - \id)$ and $\Im(V - \id)$ are both at least rank $1$. Then if $\Im(U - \id) + \Im(V - \id)$ is not rank $2$, it must be that $\Im(U - \id)$ and $\Im(V - \id)$ are subsets of a common $\bZ w_1$, where by taking out common factors we may assume $\bZ w_1 = (\bQ w_1) \cap \bZ^2$. In particular $w_1$ is an eigenvector for both $U$ and $V$. \par
    Completing $w_1$ to a basis $\{w_1, w_2\}$ for $\bZ^2$, we see that in this basis $f(U)$ and $f(V)$ are matrices of the form
    \[\begin{bmatrix} 1 & * \\ 0 & 1\end{bmatrix}\]
    inside of $\PSL_2(\bZ)$ which commute, and hence cannot freely generate a subgroup.
\end{proof}

\begin{lem}\label{lem:cech cohomology counterexample for degree = 4}
    There exists a nested sequence of subgroups
    \[\cB_4' = G_0 \supseteq G_1 \supseteq G_2 \supseteq\cdots \]
    such that the pro-group $\ul{G}$ given by the inverse system of inclusions
    \[\begin{tikzcd}
	G_0 & G_1 & G_2 & \cdots
	\arrow[from=1-2, to=1-1]
	\arrow[from=1-3, to=1-2]
	\arrow[from=1-4, to=1-3]
    \end{tikzcd}\]
    does not satisfy $(*_4)$, but has trivial dual group $\varinjlim \Hom(\ul{G}, \bZ) = 0$.
\end{lem}
\begin{proof}
    Recall that $\cB_4'$ is generated by four elements: $u, v, a, b$, such that $a$ and $b$ freely generate a normal subgroup $T$ of $\cB_4'$. The elements $u$ and $v$ also freely generate a subgroup $J$. We define the pro-group $\ul{G}$ by picking an injection $\bZ * \bZ \to (\bZ * \bZ)'$ such as the subgroup generated by $xyx^{-1}y^{-1}$ and $xy^2x^{-1}y^{-2}$, with which we can recursively define
    \begin{align*}
    u_0 = u \quad & v_0 = v \\
    u_n = u_{n - 1}v_{n - 1}u_{n - 1}^{-1}v_{n - 1}^{-1} \quad & v_n = u_{n - 1}v_{n - 1}^2u_{n - 1}^{-1}v_{n - 1}^{-2} \quad \text{for} \ n \geq 1
    \end{align*}
    to get the subgroups $J_n = \sgen{u_n, v_n} \seq J$ and $G_n = TJ_n \seq \cB_4'$.  \par

    First note that the standard inclusions $G_n \hookrightarrow \cB_4$ induce a morphism $\ul{\phi}: \ul{G} \to \cB_4$ which do not factor through $M_{4, i} \to \cB_4$. This is because all of the subgroups $G_n$ contain $T = \sgen{a, b} = \sgen{\sigma_3 \sigma_1^{-1}, \sigma_2 \sigma_1^{-1}\sigma_3 \sigma_2^{-1}}$ whose image under the canonical map $\cB_4 \to S_4$ is the subgroup $\sgen{(12)(34), (13)(24)}$ of $S_4$. Since no $i \in \{1,2,3,4\}$ is a common fixed point of these permutations, $T$ (and hence none of the $G_n$) lie in any $M_{4, i}$, so $\ul{G}$ doesn't satisfy $(*_4)$. \par

    We want to understand the groups $\Hom(G_n, \mathbb{Z})$, which comes down to understanding how $u_n$ and $v_n$ act on $a$ and $b$. Since $T$ is a normal subgroup, conjugation by an element of $J$ descends to an action on the abelianization $\Ab(T) \cong \bZ^2$. By using the images of $a$ and $b$ as a basis for $\Ab(T)$, we get a map $\alpha: J \to \Aut(\Ab(T)) \cong \GL_2(\bZ)$ where $\alpha(x) \in \Aut(\Ab(T))$ is the automorphism defined by
    \[\alpha(x)(y T') = xyx^{-1} T'.\]
    Using the relations for $\cB_4'$ we stated above we find
    \[uau^{-1} = b \mod T'\;\;\;\;\;\; \text{ and } \;\;\; \quad ubu^{-1} = a^{-1}b^3 \mod T'\]
    so the matrix $U := \alpha(u)$ in the $a,b$ basis is given by
    \[\begin{bmatrix}0 & -1 \\ 1 & 3\end{bmatrix}.\]
    Similarly for conjugation by $v$, 
    \[vav^{-1} = a^{-1}b \mod T' \;\;\;\;\;\; \text{ and } \;\;\; \quad vbv^{-1} = a^{-5}b^4 \mod T'\]
    so the matrix $V := \alpha(v)$ is given by
    \[\begin{bmatrix}-1 & -5 \\ 1 & 4\end{bmatrix}.\]
    Note this computation shows the image of $\alpha$ lies in $\SL_2(\bZ)$. We claim that $U$ and $V$ freely generate a subgroup of $\SL_2(\bZ)$, and we will use this to prove that $\varinjlim \Hom(\ul{G}, \bZ) = 0$. Note that if $U_0 = U$ and $V_0 = V$ freely generate a subgroup, then by induction so do
    \begin{align*}
    U_n & = \alpha(u_n) = U_{n - 1}V_{n - 1}U_{n - 1}^{-1}V_{n - 1}^{-1} \quad \text{ and }\\
    V_n & = \alpha(v_n) = U_{n - 1}V_{n - 1}^2U_{n - 1}^{-1}V_{n - 1}^{-2}.
    \end{align*}
    Before proving the claim, let us show that all of the structure maps in $\Hom(\ul{G}, \bZ)$ are zero (and hence the limit is zero). By definition the subgroup $G_n$ is generated by $T$ and the elements $u_n$ and $v_n$. Therefore to understand a structure map
    \[\begin{tikzcd}
	{\Hom(G_{n - 1}, \bZ)} & {\Hom(G_n, \bZ)}
	\arrow[from=1-1, to=1-2]
    \end{tikzcd}\]
    induced by the inclusion of subgroups, we need to take a morphism $f: G_{n - 1} \to \bZ$ and compute its value on $T$ and $u_n$ and $v_n$. We see that $f$ automatically vanishes on $u_n$ and $v_n$ as they are commutators of elements in $G_{n - 1}$. Since $f$ maps into the abelian $\bZ$, its restriction to $T$ factors through $\Ab(T) \cong \bZ^2$, and to show it is the zero map it is enough to show that $\tilde{f}: \Ab(T) \cong \bZ^2  \to \bZ$ has a rank $2$ subgroup in its kernel. Given any $x \in J_n$ and $y \in T$ we have
    \[\tilde{f}(y T') = \tilde{f}(x  T') + \tilde{f}(y  T') - \tilde{f}(x  T') = \tilde{f}(xyx^{-1}  T') = \tilde{f}(\alpha(x)(y  T'))\]
    showing that $\tilde{f}$ is zero on any elements $(\alpha(x) - \id)y$. In particular the images of both $U_n - \id$ and $V_n - \id$ are in the kernel of $\tilde{f}$. Then by the claim we know $U_n$ and $V_n$ freely generate a subgroup of $\SL_2(\bZ)$, so by Lemma \ref{lem:action of free subgroups of sl2Z} the sum of these images is a rank $2$ subgroup of $\bZ^2$. \par

    Finally we need to show that
    \[U = \begin{bmatrix}0 & -1 \\ 1 & 3\end{bmatrix} \quad \text{ and } \quad V = \begin{bmatrix}-1 & -5 \\ 1 & 4\end{bmatrix}\]
    freely generate a subgroup of $\SL_2(\bZ)$, which by Lemma \ref{lem:action of free subgroups of sl2Z} is the same as showing that their images in $\PSL_2(\bZ)$ freely generate a subgroup. It is a standard result that $\PSL_2(\bZ)$ has the following presentation (see \cite{PSL2Z}):
    \[\PSL_2(\bZ) = \sgen{S, Q \ | \ S^2 = 1 , \quad Q^3 = 1} \cong \bZ/2\bZ * \bZ/3\bZ\]
    where the matrices $S$ and $Q$ are given by
    \[S = \begin{bmatrix}0 & -1 \\ 1 & 0\end{bmatrix} \quad Q = \begin{bmatrix}1 & -1 \\ 1 & 0\end{bmatrix}\]
    and our matrices are $U = S(QS)^3$ and $V = SQ^{-1}SQU$. We shall compute the subgroup $\sgen{U, V}$ by using covering space theory, which will along the way show that $\sgen{U, V}$ is the commutator subgroup of $\PSL_2(\bZ)$.
    
    \begin{center}
        \begin{figure}[t]
            \includegraphics[width=200pt]{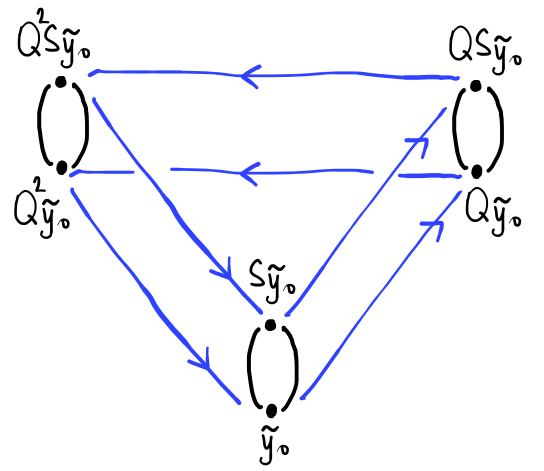}
            \centering
            \caption{Illustrated is the $1$-skeleton for $\wt{Y}$ along with labels for its six $0$-cells. The $1$-cells associated to $S$ are drawn in black, while the $1$-cells associated to $Q$ are drawn in blue.}
            \label{fig:Y_tilde}
        \end{figure}
    \end{center}

    To start note that when abelianizing $\PSL_2(\bZ)$ the element $U$ is congruent to $S^4Q^3 = 1$ and $V$ is congruent to $S^2U = 1$, so $\sgen{U, V}$ is a subgroup of the commutator subgroup of $\PSL_2(\bZ)$.
    We can form a CW complex $Y$ with a basepoint $y_0$, two $1$-cells which we label $S$ and $Q$, and two $2$-cells implementing the relations $S^2 = 1$ and $Q^3 = 1$. Thus $\pi_1(Y, y_0) \cong \PSL_2(\bZ)$ via the same presentation as above, and we can consider the covering space $(\wt{Y}, \wt{y_0})$ corresponding to the commutator subgroup of $\pi_1(Y, y_0)$. \par
    The covering map $(\wt{Y}, \wt{y_0})$ to $(Y, y_0)$ is a $6$-fold covering as the abelianization of $\PSL_2(\bZ)$ is $\bZ/2\bZ \oplus \bZ/3\bZ$, an order $6$ group. Giving $(\wt{Y}, \wt{y_0})$ the CW structure induced by the covering map, the $1$-skeleton of $(\wt{Y}, \wt{y_0})$ is as in Figure \ref{fig:Y_tilde}, as the Deck transformations for this space are $\bZ/2\bZ \oplus \bZ/3\bZ$. Since $\sgen{U, V}$ is a subgroup of the commutator subgroup, the two elements $U$ and $V$ still represent loops in $\wt{Y}$ based at $\wt{y_0}$. Collapsing the six $2$-cells associated to $Q^3 = 1$ to two points (one point by identifying $\wt{y_0}$, $Q\wt{y_0}$, $Q^2\wt{y_0}$ and the other by identifying $S\wt{y_0}$, $QS\wt{y_0}$, $Q^2S\wt{y_0}$) and collapsing the six $2$-cells associated to $S^2 = 1$ to three line segments (one between $\wt{y_0}$ and $S\wt{y_0}$, another between $Q\wt{y_0}$ and $SQ\wt{y_0}$, and the final between $Q^2\wt{y_0}$ and $SQ^2\wt{y_0}$) we see that $(\wt{Y}, \wt{y_0})$ is homotopy equivalent to the theta space and hence has $\pi_1(\wt{Y}, \wt{y_0}) \cong \bZ * \bZ$. Moreover, following the loops $U$ and $V$ along this homotopy equivalence, we see that they end up as generators for the fundamental group of the theta space (see Figure \ref{fig:U_loop}), and hence $\sgen{U, V} = \pi_1(\wt{Y}, \wt{y_0}) \cong \bZ * \bZ$ with $U$ and $V$ as generators as required.
\end{proof}

\begin{center}
        \begin{figure}[t]
            \includegraphics[width=200pt]{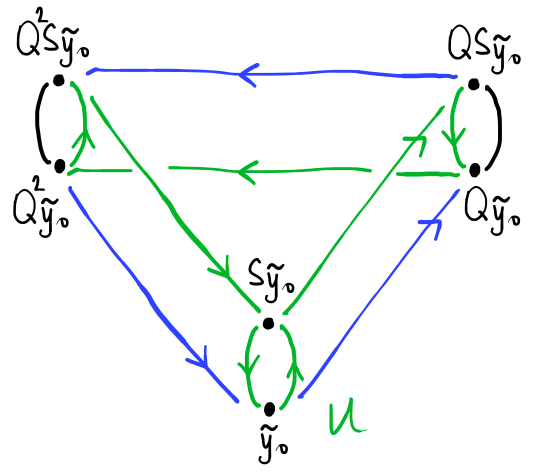}
            \includegraphics[width=200pt]{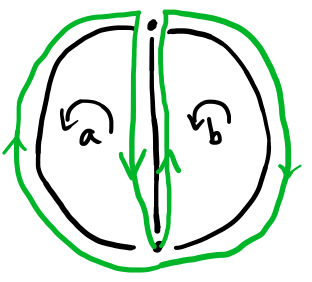}
            \centering
            \caption{On the left we illustrate the loop on $\wt{Y}$ corresponding to $U$ that starts at $\wt{y_0}$. On the right we depict this loop once we collapse $\wt{Y}$ to the theta space. If we use $a$ and $b$ to denote the loops on the theta space by traversing the respective semi-circles counterclockwise, we see that the loop corresponding to $U$ is $ab^{-1}$. A similar calculation shows that $V$ corresponds to $bab^{-1}$, and these two loops are freely generate the fundamental group.}
            \label{fig:U_loop}
        \end{figure}
    \end{center}

\begin{example}\label{ex:acyclic-2d-continua-does-not-solve-4deg-polys}
    Applying Lemma \ref{lem:compactumWithGivenProgroup}, we can find a 2-dimensional continuum $X$ for which the pro-group $\ul{G}$ in Lemma \ref{lem:cech cohomology counterexample for degree = 4} is realized as $\ppi_1(X,x_0)$. Since $\ppi_1(X,x_0) \cong \ul{G}$ does not satisfy $(*_4)$, it follows that there is a degree 4 polynomial over $C(X)$ with no repeated roots that does not possess an exact root. Moreover, since $\varinjlim \Hom(\ul{G}, \bZ) = 0$, it follows by the UCT that
    \[\check{H}^1(X) = \varinjlim \ul{H}^1(X) \cong \varinjlim \Hom(\ppi_1(X, x_0), \bZ) = 0\]
    so $X$ is acyclic. 
\end{example}

\bibliographystyle{amsalpha}
\bibliography{main}
\end{document}